\documentclass{amsart}

\usepackage[backend=biber,maxbibnames=5,maxalphanames=5,style=alphabetic,bibencoding=utf8,giveninits,url=false,isbn=false]{biblatex}
\AtBeginBibliography{\small}

\usepackage{verse}

\newlength\lineindent

\usepackage[all,hyperref]{bi-discrete}
\usepackage{cleveref}
\usepackage{tikz}
\usetikzlibrary{calc}

\usepackage{mathtools}
\usepackage{algorithm2e}

\providecommand{\ds}{\, \mathrm{d}s}
\providecommand{\dx}{\, \mathrm{d}x}

\providecommand{\dy}{\, \mathrm{d}y}

\providecommand{\tria}{\mathcal{T}}

\providecommand{\nodes}{\mathcal{N}}

\providecommand{\faces}{\mathcal{F}}
\providecommand{\cJ}{\mathcal{J}}
\providecommand{\cE}{\mathcal{E}}

\providecommand{\Inc}{\mathcal{I}_\textup{NC}}
\providecommand{\osc}{\textup{osc}}

\usepackage{dsfont}
\usepackage{pgfplots}
\usepackage{tikz}
\usepackage{listings}
\usepackage{stmaryrd}

\usepackage{tikz}
\usepackage{pgfplots}
\usepackage{pgfplotstable} 
\pgfplotsset{compat=1.18} 

\pgfplotsset{
  width=.65\linewidth,
  axis background/.style={fill=black!5!white},
  grid style={densely dotted,semithick},
  legend style={
    legend columns=1,
    legend pos=outer north east
  },
  compat=newest 
}
\pgfplotscreateplotcyclelist{MyCol1}{%
    {red,mark = *,every mark/.append style={solid,scale=0.4,fill=red}},
	{dashed,blue,mark = square,every mark/.append style={solid,scale=0.4,fill=blue}},
    {dotted,black,mark = x,every mark/.append style={solid,scale=0.8,fill=black}}
    }
\pgfplotscreateplotcyclelist{MyCol1b}{%
    {red},
	{dashed,blue},
    {dotted,black}
    }

\pgfplotscreateplotcyclelist{MyColors}{%
    {red,mark = *,every mark/.append style={solid,scale=0.4,fill=red}},
	{dotted,red,mark = *,every mark/.append style={solid,scale=0.4,fill=red}},
    {teal,mark = x,every mark/.append style={solid,scale=0.4,fill=teal}},
    {blue,mark = square*,every mark/.append style={solid,scale=0.4,fill=blue}},
    {dotted,teal,mark = x,every mark/.append style={solid,scale=0.4,fill=teal}},
    {dotted,blue,mark = square*,every mark/.append style={solid,scale=0.4,fill=blue}},
{dash dot,red,mark = *,every mark/.append style={solid,scale=0.4,fill=red}},
    {dash dot,teal,mark = x,every mark/.append style={solid,scale=0.4,fill=teal}},
    {dash dot,blue,mark = square*,every mark/.append style={solid,scale=0.4,fill=blue}},}

\makeatletter
\@namedef{subjclassname@2020}{%
  \textup{2020} Mathematics Subject Classification}
\makeatother

\begin{document}
\lstset{language=Python,
basicstyle=\small, 
keywordstyle=\color{black}\bfseries, 
commentstyle=\color{blue}, 
stringstyle=\ttfamily, 
showstringspaces=false,
numbers=left, 
numberstyle=\small, 
numbersep=10 pt,
xleftmargin= 27pt,
}

\author[J.~Storn]{Johannes Storn}
\address[J.~Storn]{Faculty of Mathematics \& Computer Science, Institute of Mathematics, Leipzig University, Augustusplatz 10, 04109 Leipzig, Germany}
\email{jstorn@math.uni-bielefeld.de}

\keywords{$p$-Laplace, Crouzeix--Raviart FEM, medius analysis, quasi-optimality, conforming companion}
\subjclass[2020]{
 	35J62, 
	65N12, 
	65N15,  
	65N30  
}

\thanks{The work of the author was supported by the Deutsche Forschungsgemeinschaft (DFG, German Research Foundation) -- SFB 1283/2 2021 -- 317210226.}

\title[Quasi-optimality of the CR FEM for the $p$-Laplacian]{Quasi-optimality of the Crouzeix--Raviart FEM for $p$-Laplace-type problems}

\begin{abstract}
We verify quasi-optimality of the Crouzeix--Raviart FEM for nonlinear problems of $p$-Laplace type. More precisely, we show that the error of the Crouzeix--Raviart FEM with respect to a quasi-norm is bounded from above by a uniformly bounded constant times the best-approximation error plus a data oscillation term. As a byproduct, we verify a novel, more localized a priori error estimate for the conforming lowest-order Lagrange FEM.
\end{abstract}
\maketitle

\section{Introduction}
While non-conforming finite element methods can offer beneficial properties compared to their conforming counterparts, such as fewer degrees of freedom, lower energy bounds, improved local approximation properties, and convergence in the presence of a Lavrentiev gap \cite{Ortner11,BalciOrtnerStorn22}, their theoretical analysis is considerably more intricate.
For example, classical a priori error analysis often requires additional smoothness of the solution; see~\cite[Sec.~2.1]{Brenner15}. This has led to the misconception that non-conforming approaches are inferior when approximating singular solutions.
For linear problems, this misconception was refuted by Gudi \cite{Gudi10,Gudi10b}; see also \cite{CarstensenSchedensack15,CarstensenGallistlSchedensack15,HuangYu21}. He introduced medius analysis -- a blend of a priori and a posteriori techniques -- which enables the verification of quasi-optimal approximation properties for non-conforming FEMs.
 
We extend the medius analysis and the resulting quasi-optimality results to nonlinear problems. In particular, we approximate the minimizer of a convex minimization problem over the space $V \coloneqq W^{1,1}_0(\Omega)$.
Given a uniformly convex N-function $\phi\colon [0,\infty) \to \mathbb{R}$ (see Definition~\ref{eq:defNfunction} below) and a right-hand side $f\in L^\infty(\Omega)$, the minimization problem  takes the form
\begin{align}\label{eq:MinProb}
u = \argmin_{v \in V} \mathcal{J}(v)\qquad\text{with }\cJ(v) \coloneqq \int_\Omega \phi(|\nabla v|) \dx - \int_\Omega fv\dx.
\end{align}
Equivalently, we want to approximate the solution $u\in V$ to the variational problem
\begin{align}\label{eq:pLaplaPrb}
\int_\Omega \frac{\phi'(|\nabla u|)}{|\nabla u|} \nabla u\cdot \nabla v\dx = \int_\Omega fv\dx\qquad\text{for all }v\in V. 
\end{align}
An example fitting into this framework is the $p$-Dirichlet energy with $\phi'(t) = (t+\kappa)^{p-2}t$ for all $t\geq 0$ with $p\in (1,\infty)$ and $\kappa \geq0$, as well as space $V = W^{1,p}_0(\Omega)$ and $f\in L^{p'}(\Omega)$ with $p' = p/(p-1)$. Setting $\kappa \coloneqq 0$ leads in \eqref{eq:pLaplaPrb} to the $p$-Laplace problem $-\divergence (|\nabla u|^{p-2}\nabla u) = f$.

We discretize the space $V$ by Crouzeix--Raviart finite elements defined as follows.
Let $\tria$ be a  regular triangulation of the domain $\Omega\subset \mathbb{R}^d$ with dimension $d>1$ in the sense of Ciarlet with faces ($d-1$ dimensional hypersurfaces) $\mathcal{F}$, interior faces $\mathcal{F}(\Omega)\coloneqq  \lbrace \gamma \in \faces\colon  \gamma \not\subset \partial \Omega\rbrace$, and boundary faces $\mathcal{F}(\partial \Omega) \coloneqq \faces \setminus \faces(\Omega)$. We denote by $\textup{mid}(\gamma)$ the barycenter of a face $\gamma \in \faces$ and define the space of piecewise affine functions and the Crouzeix--Raviart finite element space by
\begin{align*}
\begin{aligned}
\mathbb{P}_1(\tria)& \coloneqq \lbrace v_h \in L^\infty(\Omega)\colon v_h|_T \text{ is an affine polynomial for all }T\in \tria\rbrace,\\
V_h &\coloneqq \lbrace v_h \in \mathbb{P}_1(\tria) \colon v_h\text{ is continuous in } \textup{mid}(\gamma)\text{ for all }\gamma\in \faces(\Omega)\\
&\qquad\qquad\qquad\quad\ \text{ and } v_h(\textup{mid}(\gamma')) = 0\text{ for all }\gamma' \in \mathcal{F}(\partial \Omega)\rbrace.
\end{aligned}
\end{align*}
With the elementwise application of the gradient $(\nabla_h v_h)|_T \coloneqq \nabla (v_h|_T)$ for all $v_h \in V_h$ and $T\in \tria$, the resulting discretized problem seeks the minimizer
\begin{align}\label{eq:CRfem}
u_h = \argmin_{v_h \in V_h} \mathcal{J}(v_h)\qquad\text{with }\cJ(v_h) \coloneqq \int_\Omega \phi(|\nabla_h v_h|)\dx - \int_\Omega fv_h\dx.
\end{align}
Our main result is the following quasi-optimality result. The statement involves the notion of distance $\lVert F(\nabla u) - F(\nabla_h \bigcdot)\rVert_{L^2(\Omega)}$ introduced and discussed in Section~\ref{sec:ConceptOfDist}, as well as a data oscillation term $\osc^2(u,f;\tria)$ defined in~\eqref{eq:DefOsc}.
\begin{theorem}[Main result]\label{thm:mainResult}
Let the integrand $\phi$ be a uniformly convex N-function. Then the solutions $u\in V$ to \eqref{eq:MinProb} and $u_h \in V_h$ to \eqref{eq:CRfem} satisfy
\begin{align*}
& \lVert F(\nabla u) - F(\nabla_h u_h) \rVert_{L^2(\Omega)}^2 \lesssim \min_{v_h \in V_h} \lVert F(\nabla u) - F(\nabla_h v_h) \rVert_{L^2(\Omega)}^2 + \osc^2(u,f;\tria). 
\end{align*}
\end{theorem}
Theorem~\ref{thm:mainResult}, verified in Section~\ref{sec:MediusAnalysis}, marks the first a priori error estimate for non-conforming approximations for $p$-Laplace type problems that bounds the error by a best-approximation error plus a higher-order data oscillation term. For the special case $\phi(t) = t^2/2$ it recovers the best-approximation result of Gudi \cite{Gudi10} in the sense that one obtains with local mesh-size $(h_\tria)|_T \coloneqq h_T \coloneqq \textup{diam}(T)$ for all $T\in \tria$ and $L^2$-orthogonal projection onto piecewise constant functions $\Pi_0$ the estimate
\begin{align*}
& \lVert \nabla u - \nabla_h u_h \rVert_{L^2(\Omega)}^2 \lesssim \min_{v_h \in V_h} \lVert \nabla u - \nabla_h v_h \rVert_{L^2(\Omega)}^2 + \lVert h_\tria (1-\Pi_0) f \rVert^2_{L^2(\Omega)}. 
\end{align*}
A consequence of Theorem~\ref{thm:mainResult} is the following localized a priori error estimate, involving the space of piecewise constant vector valued functions 
\begin{align*}
\mathbb{P}_0(\tria;\mathbb{R}^d) \coloneqq \lbrace \chi_h \in L^\infty(\Omega;\mathbb{R}^d)\colon \chi_h|_T \in \mathbb{R}^d\text{ for all }T\in \tria\rbrace.
\end{align*}
\begin{theorem}[A priori error estimate -- non-conforming]\label{thm:Apriori}
For integrands $\phi$ that are uniformly convex N-functions the solutions $u\in V$ to \eqref{eq:MinProb} and $u_h \in V_h$ to \eqref{eq:CRfem} satisfy
\begin{align*}
 \lVert F(\nabla u) - F(\nabla_h u_h) \rVert_{L^2(\Omega)}^2& \lesssim \min_{\chi_h \in \mathbb{P}_0(\tria;\mathbb{R}^d)} \lVert F(\nabla u) - \chi_h \rVert_{L^2(\Omega)}^2 + \osc^2(u,f;\tria).
 \end{align*}
\end{theorem}
We verify Theorem~\ref{thm:mainResult} in Section~\ref{sec:MediusAnalysis} by combining medius analysis with the concept of shifted N-functions introduced and investigated by Diening and co-authors in~\cite{DieningEttwein08,DieningKreuzer08,DieningFornasierTomasiWank20}.
A brief summary of the latter is displayed in Section~\ref{sec:ConceptOfDist}.
A major difficulty that occurs in our analysis are tangential jump contributions, which we successfully treat by extending a posteriori techniques for the linear case such as in \cite{CarstensenHu07,CarstensenPuttkammer20} to the nonlinear setting.
A byproduct of our analysis is the following estimate, verified in Section~\ref{sec:Confomring}, for the conforming method
\begin{align}\label{eq:LagrangeFEM}
u_h^c = \argmin_{v_h^c\in S_0^1(\tria)} \cJ(v_h^c)\qquad\text{with }S_0^1(\tria) \coloneqq \mathbb{P}_1(\tria) \cap W^{1,1}_0(\Omega).
\end{align}
\begin{theorem}[A priori error estimate -- conforming]\label{thm:AprioriConf}
For integrands $\phi$ that are uniformly convex  N-functions the solutions $u\in V$ to \eqref{eq:MinProb} and $u_h^c \in S^1_0(\tria)$ to \eqref{eq:LagrangeFEM} satisfy
\begin{align*}
 \lVert F(\nabla u) - F(\nabla u^c_h) \rVert_{L^2(\Omega)}^2& \lesssim  \min_{\chi_h \in \mathbb{P}_0(\tria;\mathbb{R}^d)} \lVert F(\nabla u) - \chi_h \rVert_{L^2(\Omega)}^2  + \osc^2(u,f;\tria).
 \end{align*}
\end{theorem}
The a priori error estimates in Theorem~\ref{thm:Apriori} and \ref{thm:AprioriConf} suggest similar approximation properties of the conforming and non-conforming FEM, which we study numerically in Section~\ref{sec:NumExp} by applying state-of-the-art adaptive iterative schemes. 
The corresponding code can be found in \cite{StornCode}.
Notice that all results remain valid in the vector-valued setting with $V = W^{1,1}_0(\Omega;\mathbb{R}^d)$. Moreover, if $V = W^{1,p}_0(\Omega)$, the assumption $f\in L^\infty(\Omega)$ can be relaxed to $f\in L^s(\Omega)$ for suitable $s\geq 1$. 

\begin{remark}[Comparison to existing estimates]
A first result on the a priori error control for non-conforming FEMs for the $p$-Laplacian goes back to Liu and Yan \cite[Thm.~4.1]{LiuYan01}. They show under the regularity assumption $u \in C^{1,\alpha}(\overline{\Omega})$ that
\begin{align*}
\lVert F(\nabla u)-F(\nabla_h u_h)\rVert_{L^2(\Omega)}^2& \lesssim \min_{v_h\in V_h} \lVert F(\nabla u)-F(\nabla_h v_h)\rVert_{L^2(\Omega)}^2 \\
&\quad + \sum_{T \in \tria} \int_{\partial T} h_T\, |F(\nabla u) - F(\nabla v_h|_T)|^2 \ds. 
\end{align*} 
The regularity assumption holds for $p$-harmonic functions \cite{Uralceva68,Evans82,BalciBehnDieningStorn25}, but fails for rough right-hand sides $f$ and re-entrant corners in $\partial \Omega$.
An alternative estimate can be found in \cite[Thm.~3]{LiuChen20}, stating with the $P_3$-conforming companion $u_3$ of $u_h$, see \cite{CarstensenLiu15}, that
\begin{align*}
&\lVert F(\nabla u)-F(\nabla_h u_h)\rVert_{L^2(\Omega)}^2 \lesssim \lVert h_\tria(f-\Pi_0 f)\rVert_{L^2(\Omega)} \lVert u \rVert_{L^p(\Omega)}\\
&\qquad + \lVert h_\tria(f-\Pi_0 f)\rVert_{L^2(\Omega)} \lVert \nabla_h (u-u_h) \rVert_{L^p(\Omega)} + \lVert F(\nabla u_h)-F(\nabla u_3)\rVert^2_{L^2(\Omega)}.
\end{align*}
This estimate does in general not provide the optimal rate of convergence.
A more recent approach by Kaltenbach \cite{Kaltenbach24} uses medius analysis and the a posteriori error analysis in \cite{DieningKreuzer08} for the conforming space $S_0^1(\tria)$. It leads to a weaker version of Theorem~\ref{thm:mainResult}, namely
\begin{align}\label{eq:estKalt}
\begin{aligned}
 \lVert F(\nabla u) - F(\nabla_h u_h) \rVert_{L^2(\Omega)}^2 &\lesssim \min_{v_h^c \in S^1_0(\tria)} \lVert F(\nabla u) - F(\nabla v_h^c) \rVert_{L^2(\Omega)}^2\\
 &\quad + \osc^2(u,f;\tria). 
\end{aligned}
\end{align}
Notice that the combination of our novel a priori error estimate in Theorem~\ref{thm:AprioriConf} and this result implies Theorem~\ref{thm:mainResult}. A downside of the analysis in \cite{Kaltenbach24} is its direct use of the results in \cite{DieningKreuzer08} without a proper treatment of tangential jumps. In contrast, the proofs in Section~\ref{sec:MediusAnalysis} consider tangential jumps, paving the way for further investigations such as efficiency and reliability of residual-based a posteriori error estimators for non-conforming methods.
\end{remark}

\section{Concept of distance}\label{sec:ConceptOfDist}
Before embarking on the a priori error analysis, we introduce a distance measure tailored to the energy in \eqref{eq:MinProb}. For the $p$-Laplacian such a distance was introduced by Barrett and Liu in \cite{BarrettLiu94} and then extended and generalized in a series of publications by Diening and co-authors \cite{DieningEttwein08,DieningKreuzer08,DieningFornasierTomasiWank20}.
One motivation behind this concept is that conforming methods compute discrete minimizers $u_h^c \in V_h^c\subset V$ of the given energy $\mathcal{J}$. Thus, with a norm $\normm{\bigcdot}$ that is equivalent to the energy distance in the sense that  $\normm{ u - v }^2 \eqsim \mathcal{J}(v) - \mathcal{J}(u)$ for all $v\in V$, the numerical scheme will automatically be quasi-optimal; that is,
\begin{align*}
\normm{ u - u_h^c}^2 \lesssim \cJ(u_h^c) - \cJ(u) \leq \cJ(v_h^c) - \cJ(u) \lesssim \normm{ u - v_h^c }^2 \quad\text{for all }v_h^c\in V_h^c.
\end{align*}
However, for general integrands like the $p$-Laplacian such a norm is often not known. Instead, one can generalize the notion of distance, leading to quasi-norms. To define this notion, we introduce  uniformly convex N-functions.
\begin{definition}[N-function]\label{eq:defNfunction}
A function $\varphi\colon \mathbb{R}_{\geq 0} \to \mathbb{R}_{\geq 0}$ is an N-function, if
\begin{enumerate}
\item $\varphi$ is continuous and convex,
\item there is a right-continuous and non-decreasing function $\varphi'\colon \mathbb{R}_{\geq 0} \to \mathbb{R}_{\geq 0}$ with $\varphi(t) = \int_0^t \varphi'(s)\ds$ that satisfies $\lim_{t\to \infty} \varphi'(t) = \infty$ as well as $\varphi'(0) = 0$ and $\varphi'(t) >0$ for all $t>0$.
\end{enumerate}
An N-function $\phi$ is called uniformly convex if there exist constants $0<c_{uc} \leq C_{uc} < \infty$ such that for all $s> 0$ and $t\in [0,s)$
\begin{align*}
c_{uc} \frac{\phi'(s) - \phi'(t)}{s-t} \leq \frac{\phi'(s)}{s} \leq  C_{uc} \frac{\phi'(s) - \phi'(t)}{s-t}.
\end{align*}
\end{definition}
Let $\phi$ be a uniformly convex N-function. Throughout this paper, we use generic constants $C < \infty$ which may vary from line to line. They  depend on the constants $c_{uc},C_{uc}$ as well as the shape regularity of $\tria$ but are otherwise independent of the underlying triangulation and $\phi$. We use the notation $\eqsim$ and $\lesssim$ to hide the constants $C$ in the sense that $a \leq C b$ is denoted by $a \lesssim b$. If $C$ depends additionally on a parameter $\delta$, we denote this by a subscript $C_\delta$.

Apart from N-functions $\phi$, we will frequently make use of shifted N-functions, cf.~\cite[Eq.~B.4]{DieningFornasierTomasiWank20}, defined for all $r,t\geq 0$ by
\begin{align}\label{eq:ShiftedNfunction}
\phi_r(t) \coloneqq \int_0^t \phi_r'(s) \ds\qquad\text{with }\phi_r'(s) \coloneqq \frac{\phi'(\max\lbrace r,s\rbrace)}{\max\lbrace r,s\rbrace}s\qquad\text{for all } s \geq 0.
\end{align}
Moreover, we define  for all $Q\in \mathbb{R}^d$ the quantities
\begin{align*}
F(Q) \coloneqq 
\begin{cases}
\sqrt{\frac{\phi'(|Q|)}{ |Q|}}Q  &\text{for }Q \neq 0,\\
 0 &\text{for }Q=0
\end{cases}\quad\text{and}\quad A(Q) \coloneqq 
\begin{cases}
\frac{\phi'(|Q|)}{ |Q|}Q  &\text{for }Q \neq 0,\\
 0 &\text{for }Q=0.
\end{cases}
\end{align*}
\begin{proposition}[Concept of distance]\label{prop:Distances}
One has for all $P,Q\in \mathbb{R}^d$ the equivalences
\begin{align*}
|F(P) - F(Q)|^2 \eqsim (A(P) - A(Q)) \cdot (P-Q) \eqsim \phi_{|Q|}(|P-Q|).
\end{align*}
Moreover, the minimizer $u\in W^{1,1}_0(\Omega)$ to \eqref{eq:MinProb} and any $v\in W^{1,1}_0(\Omega)$ satisfy
\begin{align*}
 \mathcal{J}(v) - \mathcal{J}(u)  \eqsim \lVert F(\nabla u ) - F(\nabla v)\rVert_{L^2(\Omega)}^2.
\end{align*}
\end{proposition}
\begin{proof}
This proposition summarizes Lemma 41 and Lemma 42 in \cite{DieningFornasierTomasiWank20}.
\end{proof}
\begin{remark}[Quasi-norm of Barrett and Liu]
Since  for all $P,Q\in \mathbb{R}^d$ one has $\max\lbrace|P-Q| , |Q|\rbrace \eqsim |P-Q| + |Q|$, the $\Delta_2$ property of $\phi$ (and its shifted version) stated in Proposition~\ref{prop:Properties}~\ref{itm:Young} below implies
\begin{align*}
\phi'_{|Q|}(|P-Q|)  \coloneqq \frac{\phi'(\max\lbrace|P-Q| , |Q|\rbrace) }{\max\lbrace|P-Q| , |Q|\rbrace} |P-Q| \eqsim \frac{\phi'(|P-Q| + |Q|)}{|P-Q| + |Q|} |P-Q|.
\end{align*}
Combining this with the property $\varphi_{|Q|}'(t)\, t \eqsim \varphi_{|Q|}(t)$ for all $t\geq 0$ shown in Proposition~\ref{prop:Properties}~\ref{itm:New} below leads to
\begin{align}\label{eq:equiNorm}
\phi_{|Q|}(|P-Q|) \eqsim \frac{\phi'(|P-Q| + |Q|)}{|P-Q| + |Q|} |P-Q|^2.  
\end{align}
This relation shows the equivalence of the notion of distance in Proposition~\ref{prop:Distances} and the classical quasi-norm introduced by Barrett and Liu in \cite{BarrettLiu94} for $\phi(t) = t^p/p$.
\end{remark}
We conclude this section by displaying a toolbox for working with uniformly convex N-functions.
\begin{proposition}[Properties of N-functions]\label{prop:Properties}
A uniformly convex N-function $\phi$ has the following properties.
\begin{enumerate}
\item Its convex conjugate $\phi^*(t) \coloneqq \max_{r\geq 0} \big(rt - \phi(r)\big)$ for all $t\geq 0$ is a uniformly convex N-function. Both $\phi$ and $\phi^*$ satisfy the $\Delta_2$ condition, which reads for constants $\Delta_2 < \infty$ and $\nabla_2 < \infty$ depending solely on the uniform convexity  
\begin{align*}
\phi(2 t) \leq \Delta_2 \, \phi(t)\quad\text{and}\quad\phi^*(2 t) \leq \nabla_2 \, \phi^*(t)\qquad\text{for all }t\geq 0.
\end{align*}
Moreover, Young's inequality holds, stating for any $\delta > 0$ the existence of a constant $C_\delta < \infty$ such that \label{itm:Young}
\begin{align*}
st \leq \delta \,\phi(s) + C_\delta \phi^*(t)\qquad\text{for all }s,t \geq 0.
\end{align*}
\item One has the triangle-type inequality \label{itm:triangleInequ}
 \begin{align*}
 \phi(|a+b|) \leq \Delta_2 \big(\phi(|a|) + \phi(|b|)\big) \qquad\text{for all }a,b\in \mathbb{R}.
 \end{align*}
\item The shifted N-functions $\phi_r$ are uniformly convex N-functions with convexity constants independent of the shift $r \geq 0$.  \label{itm:ShitedN}
\item  For all $Q,P\in \mathbb{R}^d$ one has \label{itm:Udal} 
\begin{align*}
&(\phi_{|Q|})^*(|A(Q) - A(P)|) \eqsim \phi_{|Q|}(|Q- P|) \\
&\quad \eqsim |F(Q) - F(P)|^2 \eqsim (A(Q)- A(P))\cdot (Q-P).
\end{align*}
\item For all $Q,P\in \mathbb{R}^d$ one has \label{itm:Udal2}
\begin{align*}
\phi_{|Q|}'(|Q-P|) \eqsim |A(Q) - A(P)|.
\end{align*}
\item For all $t\geq 0$ one has $\phi'(t)t \eqsim \phi(t)$. \label{itm:New}
\item One has the shift-change estimate, stating for any $\delta>0$ the existence of a constant $C_\delta < \infty$ such that for all $P,Q\in \mathbb{R}^d$ and $t \geq 0$ \label{itm:shiftChange}
\begin{align*}
\phi_{|P|}(t) &\leq (1+C_\delta) \phi_{|Q|}(t) + \delta\, |F(Q)  - F(P)|^2,\\
(\phi_{|P|})^*(t) &\leq (1+C_\delta) (\phi_{|Q|})^*(t) + \delta\, |F(Q)  - F(P)|^2.
\end{align*}
\item For any $v\in W^{1,1}(T)$ with $T\in \tria$ one has the \Poincare-type inequality \label{itm:Poincare}
\begin{align*}
\min_{v_T \in \mathbb{R}}\int_T \phi(|v -v_T|)\dx \eqsim \int_T \phi(|v - \Pi_0 v|)\dx  \lesssim \int_T \phi(h_T\,|\nabla v|)\dx. 
\end{align*}
\end{enumerate}
\end{proposition}
\begin{proof}
The properties in \ref{itm:Young} are shown in \cite[Lem.~31]{DieningFornasierTomasiWank20} and \cite[Eq.~2.3]{DieningKreuzer08}.
The triangle inequality in \ref{itm:triangleInequ} follows for all $a,b\in \mathbb{R}$ from the inequality $|a+b| \leq 2 \max\lbrace |a|,|b|\rbrace$ and the $\Delta_2$ condition, implying 
\begin{align*}
\phi(|a + b|) \leq \phi(2|a|) + \phi(2|b|) \leq \Delta_2 \big( \phi(|a|) + \phi(|b|) \big).
\end{align*}
The property \ref{itm:ShitedN} is shown in \cite[Lem.~37]{DieningFornasierTomasiWank20},
\ref{itm:Udal} in \cite[Cor.~6]{DieningKreuzer08}, \ref{itm:Udal2} in \cite[Lem.~41]{DieningFornasierTomasiWank20}, and \ref{itm:New} in \cite[Eq.~2.6]{DieningKreuzer08}. The \Poincare-type inequality 
\begin{align*}
\min_{v_T \in \mathbb{R}}\int_T \phi(|v -v_T|)\dx  \lesssim \int_T \phi(h_T\,|\nabla v|)\dx. 
\end{align*}
is proven in \cite[Lem.~3.1]{DieningRuzicka07}. The shift-change \ref{itm:shiftChange} is shown in \cite[Cor.~44]{DieningFornasierTomasiWank20}. The equivalence in \ref{itm:Poincare} follows by the triangle inequality \ref{itm:triangleInequ} and Jensen's inequality in the sense for any $v_T \in \mathbb{R}$ 
\begin{align*}
\int_T \phi(|v - \Pi_0 v|)\dx & \lesssim \int_T \phi(|v - v_T|)\dx + \int_T \phi(|\Pi_0(v - v_T)|)\dx \\
& \leq 2\int_T \phi(|v - v_T|)\dx.\qedhere
\end{align*}
\end{proof}

\section{Companion operator}\label{sec:Companion}
Key tools in medius analysis are conforming companion operators, see for example~\cite[Prop.~2.3]{CarstensenGallistlSchedensack15}, mapping non-conforming functions onto conforming ones. 
For our purposes, the $\mathbb{P}_1$-conforming companion $\mathcal{E}$, also known as enriching operator \cite{Brenner96,Brenner15}, suffices. It maps the Crouzeix--Raviart space onto the lowest-order Lagrange space
$S^1_0(\tria)\coloneqq \mathbb{P}_1(\tria) \cap W^{1,\infty}_0(\Omega)$.
Let $\nodes(\Omega)$ denote the set of interior vertices in $\tria$. 
The operator $\mathcal{E}\colon V_h\to S^1_0(\tria)$ is defined via averaging the nodal values in the sense that for all $v_h \in V_h$ and interior vertices $j \in \nodes(\Omega)$ with vertex patch $\tria_j \coloneqq \lbrace T\in \tria\colon j \in T\rbrace$ we set
\begin{align}\label{eq:DefEnrich}
(\mathcal{E} v_h)(j) \coloneqq \frac{\max\lbrace v_h|_T(j) \colon T \in \tria_j\rbrace+\min\lbrace v_h|_T(j) \colon T \in \tria_j\rbrace}{2}.
\end{align}
To conclude approximation properties of $\mathcal{E}$, we recall the local mesh-size
$h_\tria\in \mathbb{P}_0(\tria)$ with $h_\tria|_T = h_T \coloneqq \textup{diam}(T)$ for all $T\in \tria$ and introduce the following notion of jumps. Given an interior face $\gamma \in \faces(\Omega)$, we fix the two simplices $T_+(\gamma) = T_+$ and $T_-(\gamma) = T_-$ in $\tria$ such that $\gamma = T_+\cap T_-$. Let $\nu\in \mathbb{R}^d$ denote the unit normal vector pointing from $T_+$ to $T_-$. Given a function $v_h \in \mathbb{P}_1(\tria)$, we define $v_h(T_\pm) \coloneqq v_h|_{T_\pm}$ and set the jump
\begin{align}\label{eq:defJump}
\llbracket v_h \rrbracket_\gamma \coloneqq v_h(T_+)|_\gamma - v_h(T_-)|_\gamma.
\end{align}
For boundary faces $\gamma \in \faces(\partial\Omega)$ let $T_+(\gamma)\in \tria$ be the simplex with $\gamma \subset T$ and set
\begin{align}\label{eq:defJump2}
\llbracket v_h \rrbracket_\gamma \coloneqq v_h(T_+(\gamma))|_\gamma  = v_h|_\gamma.
\end{align}%
The definition extends the vector-valued functions by applying it component-wise. We set the tangential trace via orthogonal projection onto the tangent space, i.e.,
\begin{align*}
  \gamma_\tau Q \coloneqq Q - (Q \cdot \nu)\,\nu = (I - \nu \otimes \nu)\,Q\qquad\text{for }Q\in \mathbb{R}^d,
\end{align*}
This is the orthogonal projection of $Q$ onto the hyperplane orthogonal to $\nu$. 
In particular, for $d=2$ and the unit tangent $\tau = (-\nu_2,\nu_1)^\top$ we have 
$\gamma_\tau Q = (Q \cdot \tau)\,\tau$, and for $d=3$ we obtain the equivalent 
representation $\gamma_\tau Q = \nu \times (Q \times \nu)$.
The trace operator allows us to decompose the jump of piecewise constant vector-valued functions $G_h \in \mathbb{P}_0(\tria;\mathbb{R}^d)$ into normal and tangential components in the sense that
\begin{align*}
|\llbracket G_h \rrbracket_\gamma |^2 = |\llbracket G_h \cdot \nu \rrbracket_\gamma|^2 + |\llbracket \gamma_\tau G_h \rrbracket_\gamma|^2.
\end{align*}
A major difficulty in the analysis of the non-conforming FEM is the occurrence of tangential jumps  $\llbracket \gamma_\tau \nabla_h v_h\rrbracket_\gamma$ with $v_h \in V_h$. The following lemma shows that the tangential jump controls the total jump $|\llbracket \nabla_h v_h\rrbracket_\gamma|$ or the normal jump of $\llbracket A(\nabla_h v_h) \cdot \nu \rrbracket_\gamma$  controls  the jump $|\llbracket A(\nabla_h u_h)\rrbracket_\gamma|$.
\begin{lemma}[Two cases]\label{lem:TwoCases}
For any $\gamma \in \faces(\Omega)$ and $v_h \in V_h$ at least one of the following two estimates holds:
\begin{align*}
|\llbracket \nabla_h v_h \rrbracket_\gamma| \eqsim |\llbracket \gamma_\tau \nabla_h v_h \rrbracket_\gamma|\qquad\text{or}\qquad |\llbracket A(\nabla_h v_h) \rrbracket_\gamma| \eqsim |\llbracket A(\nabla_h v_h)\cdot\nu \rrbracket_\gamma|.
\end{align*}
\end{lemma}
\begin{proof}
Let $\gamma \in \faces(\Omega)$ and $v_h \in V_h$.
If $\llbracket \nabla_h v_h \rrbracket_\gamma = 0$, the statement is trivially satisfied. Hence, we focus on the case $\llbracket \nabla_h v_h \rrbracket_\gamma \neq 0$.
We decompose the jump $\llbracket \nabla_h v_h \rrbracket_\gamma \in \mathbb{R}^d$ into a normal and tangential component; that is, there is a constant $\rho \in [0,1]$ and a normalized tangential vector $\tau\in \mathbb{R}^d$ with respect to $\gamma$ such that
\begin{align}\label{eq:ProofNormal}
\frac{\llbracket \nabla_h v_h \rrbracket_\gamma}{|\llbracket \nabla_h v_h \rrbracket_\gamma|} = \rho \nu  + \sqrt{1-\rho^2}\,  \tau\quad \text{or}\quad \frac{\llbracket \nabla_h v_h \rrbracket_\gamma}{|\llbracket \nabla_h v_h \rrbracket_\gamma|} = -\rho \nu  + \sqrt{1-\rho^2}\, \tau.
\end{align} 
Due to Proposition~\ref{prop:Properties}~\ref{itm:New}, \ref{itm:Udal2}, and \ref{itm:Udal} one has
\begin{align*}
&|\llbracket A(\nabla_h v_h) \rrbracket_\gamma| \eqsim \phi_{|\nabla v_h(T_+)|}'(|\llbracket \nabla_h v_h\rrbracket_\gamma|) \eqsim \frac{\phi_{|\nabla v_h(T_+)|}(|\llbracket \nabla_h v_h\rrbracket_\gamma|)}{|\llbracket \nabla_h v_h\rrbracket_\gamma|}\\
&\quad \eqsim \frac{|\llbracket A(\nabla_h v_h) \rrbracket_\gamma \cdot \llbracket \nabla_h v_h \rrbracket_\gamma|}{|\llbracket \nabla_h v_h \rrbracket_\gamma|} \leq \rho\, |\llbracket A(\nabla_h v_h) \rrbracket_\gamma\cdot \nu| + \sqrt{1-\rho^2}\, |\llbracket A(\nabla_h v_h) \rrbracket_\gamma\cdot \tau|\\
& \quad \leq  |\llbracket A(\nabla_h v_h)\cdot \nu \rrbracket_\gamma| +  |\llbracket \gamma_\tau A(\nabla_h v_h) \rrbracket_\gamma| \leq \sqrt{2}\, |\llbracket A(\nabla_h v_h) \rrbracket_\gamma|.
\end{align*}
In particular, there exists a uniformly bounded constant $K \lesssim 1$ such that
\begin{align}\label{eq:K}
|\llbracket A(\nabla_h v_h) \rrbracket_\gamma| \leq K \rho\, |\llbracket A(\nabla_h v_h) \cdot \nu \rrbracket_\gamma| + K \sqrt{1-\rho^2}\, |\llbracket \gamma_\tau A(\nabla_h v_h) \rrbracket_\gamma|.
\end{align}

\textit{Case 1.}
Suppose that
\begin{align}\label{eq:CaseOne}
(2K)^{-1} |\llbracket A(\nabla_h v_h)\rrbracket_\gamma| \leq |\llbracket A(\nabla_h v_h) \cdot \nu\rrbracket_\gamma| \leq   |\llbracket A(\nabla_h v_h)\rrbracket_\gamma|.
\end{align}
Then the lemma is satisfied.

\textit{Case 2.} Suppose that the first inequality in \eqref{eq:CaseOne} is not satisfied; that is, 
\begin{align*}
|\llbracket A(\nabla_h v_h) \cdot \nu\rrbracket_\gamma| < (2K)^{-1} |\llbracket A(\nabla_h v_h)\rrbracket_\gamma|.
\end{align*}
This assumption and \eqref{eq:K} yield
\begin{align*}
|\llbracket A(\nabla_h v_h)\rrbracket_\gamma| &\leq K\, |\llbracket A(\nabla_h v_h) \cdot \nu \rrbracket_\gamma| + K\sqrt{ 1-\rho^2}\, |\llbracket \gamma_\tau A(\nabla_h v_h)\rrbracket_\gamma| \\
&\leq 1/2\,  |\llbracket A(\nabla_h v_h) \rrbracket_\gamma| + K\sqrt{ 1-\rho^2}\, |\llbracket \gamma_\tau A(\nabla_h v_h)\rrbracket_\gamma|.
\end{align*}
Absorbing the first term on the right-hand side implies that
\begin{align*}
1/2\, |\llbracket A(\nabla_h v_h)\rrbracket_\gamma| \leq K\sqrt{ 1-\rho^2}\, |\llbracket \gamma_\tau A(\nabla_h v_h)\rrbracket_\gamma| \leq K\sqrt{ 1-\rho^2}\, |\llbracket A(\nabla_h v_h)\rrbracket_\gamma|.
\end{align*}
Hence, we have $(2K)^{-1} \leq \sqrt{ 1-\rho^2}$.
Multiplying \eqref{eq:ProofNormal} with $\tau$ thus leads to
\begin{align*}
(2K)^{-1} |\llbracket \nabla_h v_h \rrbracket_\gamma|  \leq \sqrt{1-\rho^2}\, |\llbracket \nabla_h v_h \rrbracket_\gamma| = \llbracket \nabla_h v_h \cdot \tau \rrbracket_\gamma  \leq |\llbracket \gamma_\tau \nabla_h v_h \rrbracket_\gamma| \leq |\llbracket \nabla_h v_h \rrbracket_\gamma|.
\end{align*}
This completes the proof.
\end{proof}
With this auxiliary result we analyze the approximation properties of the conforming companion operator. The result involves the integral mean $\dashint_\omega\bigcdot \dx \coloneqq |\omega|^{-1} \int_\omega\bigcdot \dx$ with Lebesgue measure $|\omega|$ for domains $\omega \subset \Omega$.
\begin{lemma}[Enrichment operator/Conforming companion]\label{lem:confComp}
For any shift $a\geq 0$, any function $v_h \in V_h$, and any simplex $T\in \tria$ the conforming companion $\mathcal{E}$ satisfies
\begin{align}\label{eq:Enrichment}
\begin{aligned}
&\dashint_T \phi_{a} (h_T^{-1} |v_h - \mathcal{E}v_h|)\dx+ \dashint_T \phi_{a} (|\nabla_h (v_h -\mathcal{E}v_h)|) \dx \\
&\qquad\qquad\qquad\qquad \qquad\lesssim \sum_{\gamma \in \faces, \gamma \cap T\neq \emptyset} \dashint_\gamma \phi_{a} (|\llbracket \gamma_\tau \nabla_h v_h\rrbracket_\gamma|)\ds.
\end{aligned}
\end{align}
Moreover, for all $u_h,v_h \in V_h$, $v\in W^{1,1}_0(\Omega)$, and $\delta > 0$ one has with $e_h \coloneqq v_h-u_h$ 
\begin{align}\label{eq:Enrichment2}
\begin{aligned}
&\int_\Omega \phi_{|\nabla_h v_h|} (h_\tria^{-1} |e_h - \mathcal{E}e_h|)\dx + \int_\Omega \phi_{|\nabla_h v_h|} (|\nabla_h (e_h -\mathcal{E} e_h)|) \dx\\
&\lesssim C_\delta\int_\Omega \phi_{|\nabla_h v_h|} (|\nabla_h e_h - \nabla  v|) \dx +\delta \sum_{\gamma \in \faces(\Omega)} \int_{\omega_\gamma} \phi_{|\nabla_h v_h|} (| \llbracket \nabla_h v_h \rrbracket_\gamma|)\dx.
\end{aligned}
\end{align}
The latter term is bounded for any $w\in W^{1,1}(\Omega)$ by
\begin{align}\label{eq:Enrichtment3}
\begin{aligned}
 \sum_{\gamma \in \faces(\Omega)} \int_{\omega_\gamma} \phi_{|\nabla_h v_h|} (| \llbracket \nabla_h v_h \rrbracket_\gamma|)\dx& \lesssim \int_\Omega \phi_{|\nabla_h v_h|} (|\nabla_h v_h - \nabla  w|) \dx \\
 &\quad +  \int_\Omega (\phi_{|\nabla_h v_h|})^*(h_\tria |\divergence A(\nabla w)|) \dx.
\end{aligned}
\end{align} 
\end{lemma}
\begin{proof}
The proof extends the proof of \cite[Eq. 3.1]{Kaltenbach24} by distinguishing the two cases in \eqref{eq:Proofpasdqq} below, leading to the best-approximation properties.

Let $a\geq 0$, $v_h \in V_h$, and $T\in \tria$. 
An inverse estimate in $L^1(T)$ and Jensen's inequality lead in combination with the $\Delta_2$ condition to  
\begin{align*}
\phi_{a} (|\nabla_h (v_h -\mathcal{E}v_h)|) &\lesssim  \phi_{a} \Big( \dashint_T h_T^{-1} | v_h -\mathcal{E}v_h| \dx \Big) \leq \dashint_T \phi_{a} (h_T^{-1} | v_h -\mathcal{E}v_h|) \dx.
\end{align*}
Hence, it suffices to bound the first term in \eqref{eq:Enrichment}. Let $\nodes(T)$ denote the set of vertices of $T$.
One has for all $x\in T$ the bound
\begin{align*}
| v_h(x) -\mathcal{E}v_h(x)| \leq \lVert v_h -\mathcal{E}v_h\rVert_{L^\infty(T)} = \max_{j\in \nodes(T)} |v_h(j) - \mathcal{E}v_h(j)|.
\end{align*}
By the definition in~\eqref{eq:DefEnrich} we can bound for each interior vertex $j\in \nodes(T) \cap \nodes(\Omega)$ the right-hand side via the sum of jumps of face neighbors in the sense that
\begin{align*}
|v_h(j) - \mathcal{E}v_h(j)| & \leq \frac{\max\lbrace v_h|_T(j) \colon T \in \tria_j\rbrace-\min\lbrace v_h|_T(j) \colon T \in \tria_j\rbrace}{2} \\
& \leq \frac{1}{2} \sum_{\gamma\in \faces(\Omega), j \in \gamma} |v_h(T_+(\gamma))(j) - v_h(T_-(\gamma))(j)|.
\end{align*}
If $j\in \nodes(T)$ is a boundary vertex in the sense that $j \in \partial\Omega$, there exists a simplex $T' \in \tria$ with face $j \in \gamma' \in \faces(T') \cap \faces(\partial\Omega)$ and one has
\begin{align*}
|v_h(j) - \mathcal{E}v_h(j)| = |v_h(j)| & \leq  |v_h(T')(j)| + \sum_{\gamma\in \faces(\Omega), j \in \gamma} |v_h(T_+(\gamma))(j) - v_h(T_-(\gamma))(j)|.
\end{align*}%
For each $\gamma \in \faces(\Omega)$ with $j \in \gamma$ the addends are bounded by 
\begin{align*}
|v_h(T_+(\gamma))(j) - v_h(T_-(\gamma))(j)| \lesssim h_T |\llbracket \gamma_\tau \nabla_h v_h \rrbracket_\gamma|.
\end{align*}
Moreover, for $T'\in \tria$ with $\gamma'\in \faces(T') \cap \faces(\partial\Omega)$ and $j \in \gamma'$ one has the upper bound
\begin{align*}
|v_h(T')(j)| = |v_h(T')(j) - v_h(T')(\textup{mid}(\gamma'))| \lesssim h_T |\llbracket \gamma_\tau \nabla_h v_h \rrbracket_{\gamma'}|.
\end{align*}%
Combining the previous estimates with the triangle inequality in Proposition~\ref{prop:Properties}~\ref{itm:triangleInequ} concludes the proof of \eqref{eq:Enrichment}.

We now extend the statement with fixed shift $a\geq 0$ to the piecewise constant function $|\nabla_h v_h|$. Recall the definition $\nabla v_h(T) \coloneqq (\nabla_h v_h)|_T \in \mathbb{R}^d$ for all $T\in \tria$.
Using \eqref{eq:Enrichment}, we obtain for $e_h \coloneqq v_h - u_h$ with $u_h \in V_h$ and $T\in \tria$ the estimate 
\begin{align}\label{eq:Proofaaaassskkk0}
\begin{aligned}
&\int_\Omega \phi_{|\nabla_h v_h|} (h_T^{-1} |e_h - \mathcal{E}e_h|)\dx + \int_\Omega \phi_{|\nabla_h v_h|} (|\nabla_h (e_h -\mathcal{E}e_h)|) \dx \\
&\quad= \sum_{T\in \tria} \int_T \phi_{|\nabla v_h(T)|} (h_T^{-1} |e_h - \mathcal{E}e_h|)\dx + \int_T \phi_{|\nabla v_h(T)|} (|\nabla_h (e_h -\mathcal{E}e_h)|) \dx\\
&\quad \lesssim \sum_{T\in \tria} \sum_{\gamma \in \faces, \gamma \cap T\neq \emptyset} h_T \int_\gamma \phi_{|\nabla v_h(T)|} (|\llbracket \gamma_\tau \nabla_h e_h\rrbracket_\gamma|)\ds.
\end{aligned}
\end{align}
For $T\in \tria$ define the element patch $\tria_T\coloneqq \lbrace T'' \in \tria\colon T\cap T''\neq \emptyset\rbrace$ and let $\gamma \in \faces$ with $\gamma \cap T \neq \emptyset$. The simplex $T_+(\gamma) \in \tria_T$ is connected to $T$ by a chain of face-neighbors $T_0\dots,T_J \in \tria_T$ in the sense that $T_0 = T$, $T_+(\gamma) = T_J$, and 
\begin{align*}
\gamma_j \coloneqq T_j \cap T_{j+1} \in \faces(\Omega)\qquad\text{for all }j=0,\dots,J-1.
\end{align*}
Such a chain exists with a uniformly bounded number $J\in \mathbb{N}_0$ of simplices due to the shape regularity of $\tria$.
An iterative application of the shift-change in Proposition~\ref{prop:Properties}~\ref{itm:shiftChange} and the equivalent notions of distances in Proposition~\ref{prop:Distances} yield for any $\delta >0$ and $\gamma \in \faces$ that
\begin{align}\label{eq:Proofaaassskkkkk-}
\begin{aligned}
h_T \int_\gamma \phi_{|\nabla v_h(T)|} (|\llbracket \gamma_\tau \nabla_h e_h\rrbracket_\gamma|)\ds &\leq C_\delta \, h_{T_+(\gamma)} \int_\gamma \phi_{|\nabla v_h(T_+(\gamma))|} (|\llbracket \gamma_\tau \nabla_h e_h\rrbracket_\gamma|)\ds\\
&\quad  + \delta\, \sum_{j=0}^{J-1}h_{T_j} \int_{\gamma_j} \phi_{|\nabla v_h(T_j)|} (|\llbracket \nabla_h v_h\rrbracket_{\gamma_j}|)\ds.
\end{aligned}
\end{align}
Combining these estimates with \eqref{eq:Proofaaaassskkk0} leads with $h_\gamma \coloneqq \textup{diam}(\gamma)$ to the bound
\begin{align}\label{eq:Proofaaaassskkk02}
\begin{aligned}
&\int_\Omega \phi_{|\nabla_h v_h|} (h_T^{-1} |e_h - \mathcal{E}e_h|)\dx + \int_\Omega \phi_{|\nabla_h v_h|} (|\nabla_h (e_h -\mathcal{E}e_h)|) \dx \\
& \qquad\qquad\qquad\qquad\lesssim C_\delta\sum_{\gamma \in \faces} h_\gamma \int_\gamma \phi_{|\nabla v_h(T_+(\gamma))|} (|\llbracket \gamma_\tau \nabla_h e_h\rrbracket_\gamma|)\ds \\
&\qquad\qquad\qquad\qquad\quad  + \delta \sum_{\gamma \in \faces(\Omega)} \int_{\omega_\gamma} \phi_{|\nabla_h v_h|} (|\llbracket  \nabla_h v_h\rrbracket_{\gamma}|)\dx.
\end{aligned}
\end{align}

We proceed with bounding the addends on the right-hand side. 
Let $\gamma \in \faces(\Omega)$ and let $b_\gamma \in W^{1,\infty}_0(\omega_\gamma;\mathbb{R}^d)$ denote a normalized vector-valued bubble function with 
\begin{align*}
\dashint_\gamma b_\gamma \ds = \frac{\llbracket \gamma_\tau\nabla_h e_h \rrbracket_\gamma}{|\llbracket \gamma_\tau\nabla_h e_h \rrbracket_\gamma|},\qquad \lVert b_\gamma \rVert_{L^\infty(\omega_\gamma)} \lesssim 1,\qquad \lVert \nabla b_\gamma \rVert_{L^\infty(\omega_\gamma)} \lesssim h_\gamma^{-1}.
\end{align*}
An elementwise integration by parts reveals for any $v \in W^{1,1}_0(\Omega)$
\begin{align}\label{eq:asdasdasdasd}
\begin{aligned}
&|\llbracket \gamma_\tau \nabla_h e_h \rrbracket_\gamma| = \dashint_{\gamma} \llbracket \gamma_\tau \nabla_h e_h \rrbracket_\gamma \cdot b_\gamma\ds = -\frac{1}{|\gamma|} \int_{\omega_\gamma} \nabla_h e_h \cdot \curl b_\gamma\dx  \\
&\qquad = \frac{1}{|\gamma|} \int_{\omega_\gamma} \nabla_h ( v-e_h) \cdot \curl b_\gamma\dx \lesssim \dashint_{\omega_\gamma} |\nabla_h ( v -e_h)|\dx. 
\end{aligned}
\end{align}
This estimate, Jensen's inequality, the shift-change in Proposition~\ref{prop:Properties}~\ref{itm:shiftChange}, and Proposition~\ref{prop:Distances} imply
\begin{align}\label{eq:Proof2asdaggaNew}
\begin{aligned}
&h_\gamma \int_{\gamma} \phi_{ |\nabla v_h(T_+(\gamma))|}(|\llbracket \gamma_\tau \nabla_h e_h\rrbracket_\gamma|) \ds \lesssim  \int_{\omega_\gamma} \phi_{ |\nabla v_h(T_+(\gamma))|}(|\nabla_h (v-e_h)|) \dx\\
&\qquad\leq C_\delta \int_{\omega_\gamma} \phi_{ |\nabla_h v_h|}(|\nabla_h (v-e_h)|) \dx + \delta \int_{\omega_\gamma}  \phi_{ |\nabla_h v_h|}(|\llbracket  \nabla_h v_h\rrbracket_\gamma|) \dx. 
\end{aligned}
\end{align}
For boundary faces $\gamma \in \faces(\partial \Omega)$ we slightly modify the arguments in \eqref{eq:asdasdasdasd} by using a bubble function $b_\gamma \in W^{1,\infty}(T_+(\gamma);\mathbb{R}^d)$ with $b_\gamma = 0$ on $\partial T_+(\gamma)\setminus \gamma$  to conclude
\begin{align}\label{eq:Proof2asdaggaNew2}
h_\gamma \int_{\gamma} \phi_{ |\nabla v_h(T_+(\gamma))|}(|\llbracket \gamma_\tau \nabla_h e_h\rrbracket_\gamma|) \ds \lesssim  \int_{\omega_\gamma} \phi_{ |\nabla_h v_h|}(|\nabla_h (v-e_h)|) \dx.
\end{align}%
Combining the results leads to \eqref{eq:Enrichment2}.

It remains to verify the best-approximation property in \eqref{eq:Enrichtment3}. For this, we fix an interior face $\gamma \in \faces(\Omega)$. 
According to Lemma~\ref{lem:TwoCases}, one has
\begin{align}\label{eq:Proofpasdqq}
|\llbracket \nabla_h v_h \rrbracket_\gamma| \eqsim |\llbracket \gamma_\tau \nabla_h v_h \rrbracket_\gamma|\qquad\text{or}\qquad |\llbracket A(\nabla_h v_h) \rrbracket_\gamma| \eqsim |\llbracket A(\nabla_h v_h)\cdot\nu \rrbracket_\gamma|.
\end{align}

\textit{Case 1.} If $|\llbracket \nabla_h v_h \rrbracket_\gamma| \eqsim |\llbracket \gamma_\tau \nabla_h v_h \rrbracket_\gamma|$, we use the result in \eqref{eq:Proof2asdaggaNew} with $e_h = v_h$ (that is $u_h = 0)$ to conclude for any $w\in W^{1,1}(\Omega)$ that 
\begin{align*}
& \int_{\omega_\gamma} \phi_{|\nabla_h v_h|} (|\llbracket  \nabla_h v_h\rrbracket_{\gamma}|)\dx \eqsim \sum_{T =T_-(\gamma),T_+(\gamma)}h_\gamma  \int_{\gamma} \phi_{ |\nabla v_h(T)|}(|\llbracket \gamma_\tau \nabla_h v_h\rrbracket_\gamma|) \ds \\ 
&\qquad \leq C_\delta \int_{\omega_\gamma} \phi_{ |\nabla_h v_h|}(|\nabla_h (w-v_h)|) \dx + \delta \int_{\omega_\gamma}  \phi_{ |\nabla_h v_h|}(|\llbracket  \nabla_h v_h\rrbracket_\gamma|) \dx.
\end{align*}
Absorbing the latter term yields the desired result
\begin{align*}
\int_{\omega_\gamma} \phi_{|\nabla_h v_h|} (|\llbracket  \nabla_h v_h\rrbracket_{\gamma}|)\dx \lesssim \int_{\omega_\gamma} \phi_{ |\nabla_h v_h|}(|\nabla_h (w-v_h)|) \dx.
\end{align*}

\textit{Case 2.} Suppose the second case in \eqref{eq:Proofpasdqq} applies; that is, $|\llbracket A( \nabla_h v_h) \rrbracket_\gamma| \eqsim |\llbracket A(\nabla_h v_h)\cdot \nu \rrbracket_\gamma|$.
This equivalence, together with Proposition~\ref{prop:Properties}~\ref{itm:Udal}, implies 
\begin{align*}
h_\gamma \int_{\gamma} \phi_{ |\nabla v_h(T_+(\gamma))|}(|\llbracket  \nabla_h v_h\rrbracket_\gamma|) \ds \eqsim h_\gamma \int_{\gamma} (\phi_{ |\nabla v_h(T_+(\gamma))|})^*(|\llbracket  A(\nabla_h v_h) \cdot \nu \rrbracket_\gamma|) \ds.
\end{align*}
Let $b_\gamma \in W^{1,\infty}_0(\omega_\gamma)$ denote a bubble function with $\dashint_\gamma b_\gamma \ds = \sgn(\llbracket A(\nabla_h v_h) \cdot \nu \rrbracket_\gamma)$, $\lVert b_\gamma \rVert_{L^\infty(\omega_\gamma)} \lesssim 1$ and $\lVert \nabla b_\gamma \rVert_{L^\infty(\omega_\gamma)} \lesssim h_\gamma^{-1}$. 
An integration by parts and Proposition~\ref{prop:Properties}~\ref{itm:Udal2} reveal for any $w\in W^{1,1}(\Omega)$ that 
\begin{align*}
&|\llbracket A(\nabla_h v_h) \cdot \nu \rrbracket_\gamma| = \dashint_\gamma \llbracket A(\nabla_h v_h) \cdot \nu \rrbracket_\gamma b_\gamma \ds =  \frac{1}{|\gamma|} \int_{\omega_\gamma} A(\nabla_h v_h) \cdot \nabla b_\gamma \dx\\
&\qquad  =  \frac{1}{|\gamma|} \int_{\omega_\gamma} \big(A(\nabla_h v_h)-A(\nabla w)\big) \cdot \nabla b_\gamma \dx - \frac{1}{|\gamma|} \int_{\omega_\gamma}b_\gamma \divergence A(\nabla w) \dx\\
&\qquad \lesssim \dashint_{\omega_\gamma} |A(\nabla_h v_h)-A(\nabla w)|\dx + \dashint_{\omega_\gamma} h_\gamma |\divergence A(\nabla w)|\dx.
\end{align*}
This bound, Jensen's inequality, and Proposition~\ref{prop:Properties}~\ref{itm:Udal} and \ref{itm:shiftChange} lead to the estimate
\begin{align*}
&h_\gamma \int_{\gamma} (\phi_{ |\nabla v_h(T_+(\gamma))|})^*(|\llbracket  A(\nabla_h v_h) \cdot \nu \rrbracket_\gamma|) \ds\\
& \lesssim  \int_{\omega_\gamma} (\phi_{ |\nabla v_h(T_+(\gamma))|})^*(|A(\nabla w) -A(\nabla_h v_h) |) \dx \\
&\quad + \int_{\omega_\gamma} (\phi_{ |\nabla v_h(T_+)|})^*(h_\gamma| \divergence(A(\nabla w))|) \dx\\
&\leq C_\delta  \int_{\omega_\gamma} \phi_{ |\nabla_h v_h|}(|  \nabla_h (w - v_h) |) \dx +C_\delta \int_{\omega_\gamma} (\phi_{ |\nabla_h v_h|})^*(h_\gamma| \divergence(A(\nabla w))|) \dx \\
&\quad + \delta \int_{\omega_\gamma}  (\phi_{ |\nabla_h v_h|})^*(|\llbracket  A(\nabla_h v_h) \rrbracket_\gamma|) \dx.
\end{align*}
With Proposition~\ref{prop:Properties}~\ref{itm:Udal} this yields
\begin{align}\label{eq:ProofFINALL}
\begin{aligned}
&\int_{\omega_\gamma} \phi_{ |\nabla_h v_h|}(|\llbracket  \nabla_h v_h\rrbracket_\gamma|) \dx \\
&\eqsim \sum_{T = T_+(\gamma),T_-(\gamma)} h_\gamma \int_{\gamma} (\phi_{ |\nabla v_h(T)|})^*(|\llbracket  A (\nabla_h v_h)\cdot \nu\rrbracket_\gamma|) \ds \\
&\lesssim C_\delta \int_{\omega_\gamma} \phi_{ |\nabla_h v_h|}(|  \nabla_h (w - v_h) |) \dx +  C_\delta \int_{\omega_\gamma} (\phi_{ |\nabla_h v_h|})^*(h_\gamma| \divergence(A(\nabla w))|) \dx\\
&\quad + \delta \int_{\omega_\gamma} \phi_{ |\nabla_h v_h|}(|\llbracket  \nabla_h v_h \rrbracket_\gamma|) \dx.
\end{aligned}
\end{align}
Since the latter term can be absorbed, combining Case 1 and Case 2 yields \eqref{eq:Enrichtment3} and concludes the proof.
\end{proof}

\section{Medius analysis}\label{sec:MediusAnalysis}
Medius analysis, introduced by Gudi in \cite{Gudi10,Gudi10b}, derives a priori error estimates by using a posteriori techniques. A typical term that occurs in this line of argument is the data oscillation. This higher-order term, cf.~Remark~\ref{rem:HigherOrder} below, is defined locally for all $v \in V+  V_h$ and $T\in \tria$ with the $L^2(\Omega)$ orthogonal projection onto piecewise constant $\Pi_0\colon L^1(\Omega) \to \mathbb{P}_0(\tria)$ by
\begin{align*}
\osc^2(v,f;T) \coloneqq \int_T (\phi_{|\nabla_h v|})^* (h_T |f - \Pi_0 f|)\dx.
\end{align*}
The global data oscillation term reads
\begin{align}\label{eq:DefOsc}
\osc^2(v,f;\tria) \coloneqq \sum_{T\in \tria} \osc^2(v,f;T).
\end{align}

\begin{lemma}[Oscillation]\label{lem:Oscillation}
Let $u\in V$ solve \eqref{eq:MinProb}. Then one has for any $T\in \tria$ and $v_h \in V_h$ the bound
\begin{align*}
\int_T (\phi_{|\nabla_h v_h|})^*(h_T\,|f|)\dx \lesssim \lVert F(\nabla u) - F(\nabla_h v_h) \rVert^2_{L^2(T)} + \osc^2(v_h,f;T).
\end{align*}
Moreover, one has
\begin{align*}
 \osc^2(v_h,f;T) \lesssim \osc^2(u,f;T) + \lVert  F(\nabla u) - F(\nabla_h v_h) \rVert_{L^2(T)}^2.
\end{align*}
\end{lemma}
\begin{proof}
This classical result can be found in \cite[Lem.~9]{DieningKreuzer08}, see also \cite[Lem.~3.2]{Kaltenbach24}. We include the instructive proof for the sake of completeness.
Let $T\in \tria$ and $v_h \in V_h$. The triangle inequality in Proposition~\ref{prop:Properties}~\ref{itm:triangleInequ} implies
\begin{align*}
\int_T (\phi_{|\nabla_h v_h|})^*(h_T\,|f|)\dx &\lesssim \int_T (\phi_{|\nabla_h v_h|})^*(h_T\,|f - \Pi_0 f|)\dx\\
&\quad + \int_T (\phi_{|\nabla_h v_h|})^*(h_T\,|\Pi_0 f|)\dx.
\end{align*}
Let $b_T \in W^{1,\infty}_0(T)$ be a bubble function with 
\begin{align*}
\dashint_T b_T \dx = \sgn((\Pi_0 f)|_T),\qquad\lVert b_T \rVert_{L^\infty(T)}\lesssim 1,\qquad \lVert \nabla b_T \rVert_{L^\infty(T)}\lesssim h_T^{-1}.
\end{align*}
The latter term satisfies due to the triangle inequality in Proposition~\ref{prop:Properties}~\ref{itm:triangleInequ} 
\begin{align}\label{eq:Proofasmmms}
\begin{aligned}
&\int_T (\phi_{|\nabla_h v_h|})^*(h_T\,|\Pi_0 f|)\dx = \int_T (\phi_{|\nabla_h v_h|})^*\Big( h_T\Big| \dashint_T (\Pi_0 f) b_T \dy\Big|\Big)\dx\\
&\lesssim  \int_T (\phi_{|\nabla_h v_h|})^*\Big( h_T\Big| \dashint_T (f-\Pi_0 f) b_T \dy\Big|\Big)\dx  + \int_T (\phi_{|\nabla_h v_h|})^*\Big( h_T\Big| \dashint_T  f b_T \dy\Big|\Big)\dx.
\end{aligned}
\end{align}
The property $\lVert b_T \rVert_{L^\infty(T)} \lesssim 1$, the $\Delta_2$ condition in Proposition~\ref{prop:Properties}~\ref{itm:Young}, and 
Jensen's inequality show for the first addend in \eqref{eq:Proofasmmms} that 
\begin{align*}
 \int_T (\phi_{|\nabla_h v_h|})^*\Big( h_T\Big| \dashint_T (f-\Pi_0 f) b_T \dy\Big|\Big)\dx &\lesssim 
 \int_T (\phi_{|\nabla_h v_h|})^*\Big( \dashint_T h_T\, | f-\Pi_0 f\, |\dy\Big)\dx \\
& \leq \int_T (\phi_{|\nabla_h v_h|})^*( h_T\, | f-\Pi_0 f\, |)\dx.
\end{align*}
An integration by parts, Proposition~\ref{prop:Properties}~\ref{itm:Udal}, and Jensen's inequality show for the second addend in \eqref{eq:Proofasmmms} that 
\begin{align*}
&\int_T (\phi_{|\nabla_h v_h|})^*\Big( h_T\Big| \dashint_T  f b_T \dy\Big|\Big)\dx\\
&\qquad = \int_T (\phi_{|\nabla_h v_h|})^*\Big( h_T\Big| \dashint_T (A(\nabla u ) - A(\nabla_h v_h)) \cdot \nabla b_T \dy\Big|\Big)\dx \\
&\qquad \lesssim \int_T (\phi_{|\nabla_h v_h|})^*\Big( \dashint_T |A(\nabla u ) - A(\nabla_h v_h)| \dy\Big)\dx \\
&\qquad\leq \int_T (\phi_{|\nabla_h v_h|})^*(|A(\nabla u ) - A(\nabla_h v_h)| )\dx.
\end{align*}
Combining these estimates and applying Proposition~\ref{prop:Properties}~\ref{itm:Udal} conclude the proof of the first bound in the lemma. The second bound follows by the shift-change in Proposition~\ref{prop:Properties}~\ref{itm:shiftChange}.
\end{proof}
After this preliminary result, we are able to verify this paper's main result.
The proof follows the ideas of \cite[Lem.~2.1]{Gudi10}; see also \cite[Sec.~2.2]{Brenner15} and \cite[Thm.~3.1]{Kaltenbach24}. The latter already extends the arguments from the linear to the nonlinear setting, but it does not separate normal and tangential jump contributions, which leads to the estimate in \eqref{eq:estKalt}.
\begin{proof}[Proof of Theorem~\ref{thm:mainResult}]
For any $v_h \in V_h$ one has
\begin{align*}
&\lVert F(\nabla u)- F(\nabla_h u_h) \rVert_{L^2(\Omega)}\\
&\qquad\qquad \leq \lVert F(\nabla u)- F(\nabla_h v_h) \rVert_{L^2(\Omega)} + \lVert F(\nabla_h u_h)- F(\nabla_h v_h) \rVert_{L^2(\Omega)}.
\end{align*}
The second term satisfies with $e_h \coloneqq v_h - u_h$
\begin{align}\label{eq:MainSplitProof}
\begin{aligned}
\lVert F(\nabla_h u_h)- F(\nabla_h v_h) \rVert_{L^2(\Omega)}^2 &\eqsim \langle A(\nabla_h v_h) - A(\nabla_h u_h), \nabla_h e_h\rangle_\Omega \\
& = \langle A(\nabla_h v_h), \nabla_h (e_h - \cE e_h)\rangle_\Omega \\
&\quad
+\langle f,\cE e_h - e_h\rangle_\Omega\\
&\quad
+\langle A(\nabla_h v_h) - A(\nabla_h u_h), \nabla \cE e_h\rangle_\Omega.
\end{aligned}
\end{align}
Let $\cE$ denote the $\mathbb{P}_1$-conforming companion introduced in Section~\ref{sec:Companion}

\textit{Step 1 (First addend in \eqref{eq:MainSplitProof}).}
The trace inequality and an inverse estimate yield for any interior face $\gamma \in \faces(\Omega)$
\begin{align*}
\Big| \dashint_\gamma e_h - \cE e_h\ds\Big| &\lesssim \dashint_{\omega_\gamma} | e_h - \cE e_h| \dx +  h_\gamma\dashint_{\omega_\gamma} |\nabla_h ( e_h - \cE e_h)| \dx \\ 
&\lesssim \dashint_{\omega_\gamma} |e_h - \cE e_h|\dx.
\end{align*}
Combining this estimate with an elementwise integration by parts leads to
\begin{align}\label{eq:sadkkgeq}
\begin{aligned}
\langle A(\nabla_h v_h), \nabla_h( e_h - \cE e_h)\rangle_\Omega & = \sum_{\gamma \in \faces(\Omega)} \llbracket A(\nabla_h v_h) \cdot \nu \rrbracket_\gamma \int_\gamma e_h- \cE e_h\ds \\
&\lesssim \sum_{\gamma \in \faces(\Omega)} \int_{\omega_\gamma} |\llbracket A(\nabla_h v_h) \cdot \nu \rrbracket_\gamma|\, h_\gamma^{-1}  |e_h- \cE e_h|\dx .
\end{aligned}
\end{align}
Proposition~\ref{prop:Properties}~\ref{itm:Udal} and Lemma~\ref{lem:confComp} imply for any $\delta >0$  that
\begin{align*}
\begin{aligned}
&\sum_{\gamma\in \faces(\Omega)} \int_{\omega_\gamma} |\llbracket A(\nabla_h v_h) \cdot \nu \rrbracket_\gamma|\,  h_\gamma^{-1}  |e_h- \cE e_h|\dx\\
& \leq \sum_{\gamma\in \faces(\Omega)} C_\delta \int_{\omega_\gamma}(\phi_{|\nabla_h v_h|})^*(|\llbracket A(\nabla_h v_h) \cdot \nu\rrbracket_\gamma|) \dx  + \delta\int_{\omega_\gamma}  \phi_{|\nabla_h v_h|}( h_\gamma^{-1}  |e_h- \cE e_h| ) \dx \\
&\leq \sum_{\gamma\in \faces(\Omega)} C_\delta \int_{\omega_\gamma}\phi_{|\nabla_h v_h|}(|\llbracket \nabla_h v_h \rrbracket_\gamma|) \dx +  \delta\, \lVert F(\nabla_h u_h) - F(\nabla_h v_h)\rVert^2_{L^2(\Omega)} \\
& \leq C_\delta \, \lVert F(\nabla u) - F(\nabla_h v_h)\rVert^2_{L^2(\Omega)} + C_\delta\int_\Omega (\phi_{|\nabla_h v_h|})^*(h_\tria\,|f|)\dx\\
&\quad  + \delta\, \lVert F(\nabla_h u_h) - F(\nabla_h v_h)\rVert^2_{L^2(\Omega)}.
\end{aligned}
\end{align*}

\textit{Step 2 (Second addend in \eqref{eq:MainSplitProof}).}
Young's inequality and Lemma~\ref{lem:confComp} show that 
\begin{align*}
\langle f , \cE e_h - e_h\rangle_\Omega & \leq C_\delta \int_\Omega (\phi_{|\nabla_h v_h|})^* (h_\tria\, |f|)\dx + \delta\, \int_\Omega \phi_{|\nabla_h v_h|}(h_\tria^{-1}|e_h -\cE e_h|)\dx  \\
&\leq C_\delta \int_\Omega (\phi_{|\nabla_h v_h|})^* (h_\tria\, |f|)\dx + \delta\, \lVert F(\nabla u) - F(\nabla_h v_h)\rVert^2_{L^2(\Omega)} \\
&\quad + \delta\, \lVert F(\nabla_h u_h) - F(\nabla_h v_h)\rVert^2_{L^2(\Omega)}.
\end{align*}

\textit{Step 3 (Third addend in \eqref{eq:MainSplitProof}).}
Since $\cE e_h \in V_h \cap V$, we have
\begin{align*}
\langle A(\nabla_h u_h), \nabla \cE e_h\rangle_\Omega = \langle f,\cE e_h\rangle_\Omega = \langle A(\nabla u), \nabla \cE e_h\rangle_\Omega.
\end{align*}
Combining this observation with Proposition~\ref{prop:Properties}~\ref{itm:Udal} and \ref{itm:Young} and Lemma~\ref{lem:confComp} reveals
\begin{align*}
&\langle A(\nabla_h v_h) - A(\nabla_h u_h), \nabla \cE e_h\rangle_\Omega = \langle A(\nabla_h v_h) - A(\nabla u), \nabla \cE e_h\rangle_\Omega\\
& = \langle A(\nabla_h v_h) - A(\nabla u), \nabla_h e_h\rangle_\Omega + \langle A(\nabla_h v_h) - A(\nabla u),\nabla \cE e_h- \nabla_h e_h\rangle_\Omega \\
& \leq C_\delta\, \lVert F(\nabla u) - F(\nabla_h v_h)\rVert_{L^2(\Omega)}^2 + \delta\, \lVert F(\nabla_h u_h) - F(\nabla_h v_h)\rVert_{L^2(\Omega)}^2\\
&\quad + \delta\, \int_\Omega \phi_{|\nabla_h v_h|} (|\nabla_h (e_h - \cE e_h)|)\dx\\
&\leq  C_\delta\, \lVert F(\nabla u) - F(\nabla_h v_h)\rVert_{L^2(\Omega)}^2 + \delta \int_\Omega (\phi_{|\nabla_h v_h|})^* (h_\tria\, |f|)\dx \\
&\quad + \delta\, \lVert F(\nabla_h u_h) - F(\nabla_h v_h)\rVert_{L^2(\Omega)}^2.
\end{align*}

\textit{Step 4 (Conclusion).}
Combining the results from Step 1 to Step 3 with \eqref{eq:MainSplitProof}, absorbing the term $\delta\, \lVert F(\nabla_h u_h) - F(\nabla_h v_h)\rVert_{L^2(\Omega)}^2$, and applying Lemma~\ref{lem:Oscillation} concludes the proof. 
\end{proof}

In order to conclude the localized a priori error estimate in Theorem~\ref{thm:Apriori}, we introduce the non-conforming interpolation operator $\Inc\colon W^{1,1}_0(\Omega) \to V_h$, defined for all $v\in W^{1,1}_0(\tria)$ via the identity 
\begin{align*}
\dashint_\gamma \Inc v \ds = \dashint_\gamma v\ds\qquad\text{for all }\gamma \in \faces(\Omega).
\end{align*}
A well-known property, see for example \cite[Lem.~2]{OrtnerPraetorius11}, is the commuting diagram
\begin{align}\label{eq:Ince}
\nabla_h \Inc v = \Pi_0 \nabla v\qquad\text{for all }v\in W^{1,1}_0(\Omega). 
\end{align}
With this property and the quasi-best-approximation result in Theorem~\ref{thm:mainResult} we can verify the a priori error estimate in Theorem~\ref{thm:Apriori}.
\begin{proof}[Proof of Theorem~\ref{thm:Apriori}]
Let $u\in V$ and $u_h \in V_h$ denote the solutions to \eqref{eq:MinProb} and \eqref{eq:CRfem}, respectively. Theorem~\ref{thm:mainResult} yields the bound
\begin{align*}
\lVert F(\nabla u) - F(\nabla_h u_h) \rVert_{L^2(\Omega)}^2 \lesssim \osc^2(u,f;\tria) + \lVert F(\nabla u) - F(\nabla_h \Inc u) \rVert_{L^2(\Omega)}^2.
\end{align*}
Using Proposition~\ref{prop:Distances}, the commuting diagram property \eqref{eq:Ince}, the triangle inequality in Proposition~\ref{prop:Properties}~\ref{itm:triangleInequ}, Jensen's inequality, and the shift-change estimate in Proposition~\ref{prop:Properties}~\ref{itm:shiftChange} we conclude for all piecewise constant functions $\chi_\tria \in \mathbb{P}_0(\tria;\mathbb{R}^d)$ that %
\begin{equation}\label{eq:approxInc}
\begin{aligned}
&\lVert F(\nabla u) - F(\nabla_h \Inc u) \rVert_{L^2(\Omega)}^2 \eqsim \int_\Omega \phi_{|\nabla_h \Inc u|}(|\nabla u - \nabla_h \Inc u|)\dx \\
&\qquad = \int_\Omega \phi_{|\nabla_h \Inc u|}(|\nabla u - \Pi_0 \nabla u|)\dx \\
&\qquad \lesssim \int_\Omega \phi_{|\nabla_h \Inc u|}(|\nabla u - \chi_\tria|)\dx + \int_\Omega \phi_{|\nabla_h \Inc u|}(|\Pi_0 (\nabla u - \chi_\tria)|)\dx \\
&\qquad \leq 2 \int_\Omega \phi_{|\nabla_h \Inc u|}(|\nabla u - \chi_\tria|)\dx \\
&\qquad\leq C_\delta\, \lVert F(\nabla u) - F(\chi_\tria)\rVert_{L^2(\Omega)}^2 + \delta\, \lVert F(\nabla u) - F(\nabla_h \Inc u) \rVert_{L^2(\Omega)}^2.
\end{aligned}
\end{equation}
Absorbing the second term concludes the proof.
\end{proof}
Let us conclude this section with a short discussion on the higher-order property of the oscillation.
\begin{remark}[Higher-order term]\label{rem:HigherOrder}
Let us assume that $\phi \in W^{2,\infty}_\textup{loc}((0,\infty))$ is almost everywhere twice differentiable and we have constants $1< p_- \leq p_+ < \infty$ such that for any shift $a \geq 0$
\begin{align}\label{eq:index}
p_- \leq \frac{\phi_a''(t)t}{\phi_a'(t)}  + 1 \leq p_+\qquad\text{for almost all }t > 0.
\end{align}
Then one has for all $a,t,s\geq 0$ the bounds \cite[Lem.~A.2]{BehnDiening24}
\begin{align}\label{eq:Behn}
\begin{aligned}
p_- &\leq (\phi_a'(t)t)/\phi_a(t) \leq  p_+,\\
\min \lbrace t^{p_+} ,t^{p_-}\rbrace \phi_a(s) &\leq \phi_a(ts) \leq \max \lbrace t^{p_+},t^{p_-}\rbrace \phi_a(s).
\end{aligned}
\end{align}
Moreover, for all $a,b,t\geq 0$ \cite[Lem.~43]{DieningFornasierTomasiWank20} and \cite[Prop.~2.2]{DieningStorn25} state that 
\begin{align*}
\phi_a'(t) &\lesssim \phi_b'(t) +  \phi_b'(|a-b|) \quad \text{and}\quad
(\phi_a)^*(\phi_a'(t)) = \phi_a'(t)t - \phi_a(t)\leq \phi_a'(t)t.
\end{align*}
Using these estimates and Young's inequality, we can quantify the constants in the shift-change (Proposition~\ref{prop:Properties}~\ref{itm:shiftChange}) via
\begin{align*}
 p_-\phi_{|P|}(t) & \leq \phi_{|P|}'(t) t \lesssim \phi'_{|Q|}(t) t + \delta (\phi'_{|Q|}(|P-Q|)t /\delta)\\
& \leq \phi'_{|Q|}(t) t + \delta \big((\phi_{|Q|})^*(\phi'_{|Q|}(|P-Q|)) + \phi_{|Q|}(t /\delta)\big) \\
& \leq \phi'_{|Q|}(t) t + \delta \phi'_{|Q|}(|P-Q|)|P-Q| + \max \lbrace \delta^{1-p_+},\delta^{1-p_-}\rbrace\phi_{|Q|}(t) \\
&\leq ( p_+ + \max \lbrace \delta^{1-p_+},\delta^{1-p_-}\rbrace)\phi_{|Q|}(t) + p_+ \delta \phi_{|Q|}(|P-Q|).
\end{align*}
Since we need this property for the convex conjugate, we exploit the identity $(\phi_{a})^*(t) = (\phi^*)_{\phi'(a)}(t)$ for all $t,a\geq 0$ \cite[Lem.~33]{DieningFornasierTomasiWank20} and the property that $\phi_a^*$ satisfies the property in \eqref{eq:index} with indices $p_-$ and $p_+$ replaced by the dual exponents $p_+'$ and $p_-'$ \cite[Lem.~A.1]{BehnDiening24}.  This leads for all $t\geq 0$ to the bound
\begin{align*}
(\phi_{|P|})^*(t) = (\phi^*)_{\phi'(|P|)}(t) & \lesssim (p_-' + \max\lbrace \delta^{1-p'_-},\delta^{1-p'_+}\rbrace)(\phi_{|Q|})^*(t)\\
&\quad + p_-'  \delta (\phi_{|Q|})^*(|A(P)-A(Q)|).
\end{align*}
Using this shift-change, Proposition~\ref{prop:Distances}, and Proposition~\ref{prop:Properties}~\ref{itm:Udal}, we obtain  for all $T\in \tria$ with some constant $C$ independent of $\delta > 0$ the bound
\begin{align*}
&\int_T (\phi_{|\nabla u|})^* (h_T |f - \Pi_0 f|) \dx \leq \delta\, \lVert F(\nabla u) - F( \Pi_0 \nabla u) \rVert_{L^2(T)}^2 \\
&\qquad\qquad\qquad\qquad + C (1 + \max \lbrace \delta^{1-p_+'},\delta^{1-p_-'}\rbrace)\int_T (\phi_{ |\Pi_0 \nabla u|})^*  (h_T |f - \Pi_0 f|) \dx. 
\end{align*}
\Poincare's inequality (Proposition~\ref{prop:Properties}~\ref{itm:Poincare}) and \eqref{eq:Behn} bound for $h_T \leq 1$ the second addend by 
\begin{align*}
\int_T (\phi_{ |\Pi_0 \nabla u|})^*  (h_T |f - \Pi_0 f|) \dx &\lesssim \int_T (\phi_{ |\Pi_0 \nabla u|})^*  (h^2_T |\nabla f|) \dx\\
& \leq h_T^{2 p_+'} \int_T (\phi_{ |\Pi_0 \nabla u|})^*  (|\nabla f|) \dx.
\end{align*}
Combining the equivalent notions of distances in Proposition~\ref{prop:Distances} with the arguments in the proof of Proposition~\ref{prop:Properties}~\ref{itm:Poincare} allows us to bound the first addend by
\begin{align*}
\lVert F(\nabla u) - F( \Pi_0 \nabla u) \rVert_{L^2(T)}^2 \lesssim \min_{\chi_h \in \mathbb{P}_0(T;\mathbb{R}^d)} \lVert  F(\nabla u) - \chi_h \rVert_{L^2(T)}^2.
\end{align*}
Choosing $\delta \coloneqq h_T^{2(1-1/p_+')} = h_T^{2/p_+}$ yields for $h_T \leq 1$ the estimate
\begin{align*}
\int_T (\phi_{|\nabla u|})^* (h_T |f - \Pi_0 f|) \dx &\lesssim h_T^{2/p_+} \min_{\chi_h \in \mathbb{P}_0(T;\mathbb{R}^d)} \lVert  F(\nabla u) - \chi_h \rVert_{L^2(T)}^2\\
& \quad + h_T^{2/p_+}  h_T^2 \int_T (\phi_{ |\Pi_0 \nabla u|})^* (|\nabla f|) \dx.
\end{align*}
This shows that for sufficiently smooth right-hand sides the oscillation converges with a higher order than the approximation error $\lVert F (\nabla u) - F(\nabla_h u_h) \rVert_{L^2(\Omega)}^2$. 
\end{remark}
\section{Conforming finite element method}\label{sec:Confomring}
By combining the non-conforming interpolation operator $\Inc$ with the conforming companion operator $\cE$, we obtain an interpolation operator onto the lowest-order Lagrange finite element space $S^1_0(\tria) \coloneqq \mathbb{P}_1(\tria) \cap W^{1,1}_0(\Omega)$ defined by
\begin{align}\label{eq:IOperator}
\mathcal{I} \coloneqq \cE \circ \Inc\colon W^{1,1}_0(\Omega)\to S^1_0(\tria).
\end{align}
This operator has the following approximation properties.
\begin{theorem}[Approximation properties]\label{thm:ApproxProperties}
For any $v\in W^{1,1}_0(\Omega)$ one has the interpolation error
\begin{align*}
&\lVert F(\nabla v) - F(\nabla \mathcal{I}v)\rVert_{L^2(\Omega)}^2\\
&\qquad \lesssim \min_{\chi_\tria \in \mathbb{P}_0(\tria;\mathbb{R}^d)} \lVert F(\nabla v) - \chi_\tria\rVert_{L^2(\Omega)}^2 + \int_\Omega (\phi_{|\nabla v|})^*(h_\tria|\divergence A(\nabla v)|)\dx. 
\end{align*}
\end{theorem}
\begin{proof}
Let $v\in W^{1,1}_0(\Omega)$. The triangle inequality reveals 
\begin{align}\label{eq:ProofTemp022}
\begin{aligned}
&\lVert F(\nabla v) - F(\nabla \mathcal{I}v)\rVert_{L^2(\Omega)}\\
&\quad \leq \lVert F(\nabla v) - F(\nabla_h \Inc v)\rVert_{L^2(\Omega)} + \lVert F(\nabla \cE \Inc v) - F(\nabla_h \Inc v)\rVert_{L^2(\Omega)}.
\end{aligned}
\end{align}
The second addend is bounded due to Proposition~\ref{prop:Distances} and Lemma~\ref{lem:confComp} by
\begin{align}\label{eq:ProofTemp0223}
\begin{aligned}
&\lVert F(\nabla \cE \Inc v) - F(\nabla_h \Inc v)\rVert_{L^2(\Omega)}\\
&\qquad \lesssim \lVert F(\nabla v) - F(\nabla_h \Inc v)\rVert_{L^2(\Omega)} + \int_\Omega (\phi_{|\nabla v|})^* ( h_\tria |\divergence A(\nabla v)|)\dx.
\end{aligned}
\end{align}
As shown in \eqref{eq:approxInc}, the first addend on the right-hand side of \eqref{eq:ProofTemp022} and \eqref{eq:ProofTemp0223} is bounded by $ \min_{\chi_\tria \in \mathbb{P}_0(\tria;\mathbb{R}^d)} \lVert F(\nabla v) - \chi_\tria\rVert_{L^2(\Omega)}^2$. Combining these results concludes the proof.
\end{proof}
\begin{remark}[$\divergence A(\nabla v)$]
In contrast to the linear case, cf.~\cite{Veeser16}, the estimate in Theorem~\ref{thm:ApproxProperties} contains the additional term $\int_\Omega \phi_{|\nabla v|}(h_\tria|\divergence A(\nabla v)|)\dx$. This is due to the fact that the shift-change in \eqref{eq:Proofaaassskkkkk-}, which is not needed in the linear case, requires control of the entire jump $\llbracket \nabla_h \bigcdot \rrbracket_\gamma$ instead of the tangential jump $\llbracket \gamma_\tau \nabla_h \bigcdot \rrbracket_\gamma$. Controlling the normal jump leads to the additional contribution. 
\end{remark}
With the operator $\mathcal{I}$ in \eqref{eq:IOperator}, we can verify the novel a priori error estimate for the lowest-order Lagrange FEM in Theorem~\ref{thm:AprioriConf}.
\begin{proof}[Proof of Theorem~\ref{thm:AprioriConf}]
Let $u\in V$ and $u_h^c \in S_0^1(\tria)$ denote the solutions to \eqref{eq:MinProb} and \eqref{eq:LagrangeFEM}, respectively.
Since $S_0^1(\tria)$ is a conforming space, established quasi-optimality results \cite[Lem.~5.2]{DieningRuzicka07} yield
\begin{align*}
\lVert F(\nabla u) - F(\nabla u^c_h) \rVert_{L^2(\Omega)}^2& \lesssim \min_{v_h^c\in S_0^1(\tria)}  \lVert F(\nabla u) - F(\nabla v^c_h) \rVert_{L^2(\Omega)}^2 \\
& \leq \lVert F(\nabla u) - F(\nabla \mathcal{I} u) \rVert_{L^2(\Omega)}^2.
\end{align*}
Theorem~\ref{thm:ApproxProperties} and the property $-\divergence A(\nabla u) = f$ bound the latter term by 
\begin{align*}
\lVert F(\nabla u) - F(\nabla \mathcal{I} u) \rVert_{L^2(\Omega)}^2 \lesssim\min_{\chi_\tria \in \mathbb{P}_0(\tria;\mathbb{R}^d)} \lVert F(\nabla u) - \chi_\tria\rVert_{L^2(\Omega)}^2  + \int_\Omega \phi_{|\nabla u|}(h_\tria |f|)\dx.
\end{align*}
The shift-change in Proposition~\ref{prop:Properties}~\ref{itm:shiftChange} and Lemma~\ref{lem:Oscillation} yield
\begin{align*}
\int_\Omega \phi_{|\nabla u|}(h_\tria |f|)\dx & \lesssim \int_\Omega \phi_{|\nabla_h \Inc u|}(h_\tria |f|)\dx + \lVert F(\nabla u) - F(\nabla_h \Inc u) \rVert_{L^2(\Omega)}^2 \\
&\lesssim \lVert F(\nabla u) - F(\nabla_h \Inc u) \rVert_{L^2(\Omega)}^2  + \osc^2(u,f;\tria).
\end{align*}
Applying \eqref{eq:approxInc} concludes the proof.
\end{proof}

\section{Numerical experiments}\label{sec:NumExp}
We conclude this paper with a numerical investigation of the lowest-order Lagrange and Crouzeix--Raviart FEM for the $p$-Laplacian; that is, the integrand $\phi(t) \coloneqq t^p/p$ for all $t\geq 0$. We solve the problem with an adaptive routine discussed below. Our implementation uses NGSolve \cite{Schoeberl97,Schoeberl14} and can be found in \cite{StornCode}.
\subsection{Adaptive scheme}
For small and large exponents $p$ typical iterative approaches such as Newton's method or gradient descent schemes fail to compute the discrete solution. However, the regularized \Kacanov{} scheme introduced in \cite{DieningFornasierTomasiWank20} for $p<2$ and modified for $p>2$ in \cite{BalciDieningStorn23,DieningStorn25} provides a method with guaranteed convergence for $p \in (1,2]$ and $p\in [2,\infty)$, respectively. The overall idea is to minimize a regularized energy $\cJ_\varepsilon(v_h) \coloneqq \int_\Omega \phi_\varepsilon(|\nabla_h v_h|) - fv_h\dx$ over all $v_h \in V_h$ by iteratively solving weighted Laplace problems. The minimizer of the regularized energy converges towards the minimizer of the energy $\cJ$ as the relaxation interval $\varepsilon = (\varepsilon_-,\varepsilon_+) \subset (0,\infty)$ is enlarged. Given an iterate $u_{h,n}\in V_h$, the impact of the regularization can  be measured by comparing the energy differences 
\begin{align*}
\eta_-^2 & \coloneqq \cJ_\varepsilon(u_{h,n}) - \cJ_{(0,\varepsilon_+)}(u_h),\\
\eta_+^2 & \coloneqq \cJ_\varepsilon(u_{h,n}) - \cJ_{(\varepsilon_-,\infty)}(u_h).
\end{align*} 
We use the discrete duality based error estimator $\eta^2_\textup{iter}$ discussed in \cite[Sec.~4]{DieningStorn25} to bound the iteration error 
\begin{align*}
\cJ_\varepsilon(u_{h,n}) - \min_{v_h \in V_h} \cJ_{\varepsilon}(v_h) \leq \eta^2_\textup{iter}.
\end{align*}
Furthermore, we use the residual based error estimator $\eta^2_\tria = \eta^2_\tria(u_{h,n},\varepsilon)$ in \cite{DieningKreuzer08}, to estimate the distance 
\begin{align*}
\cJ_\varepsilon(u_{h,n}) -  \min_{v \in V} \cJ_{\varepsilon}(v) + \osc^2(u_{h,n},f;\tria) \eqsim \eta^2_\tria.
\end{align*} 
However, for $p \gg 2$ this estimator leads to the refinements of very few simplices per iteration, indicating the struggles of this estimator when applied to inexact discrete solutions and large values of $p$. Instead, we use for  $p>2$ the dual discrete iterate $\sigma_{h,n}\in \mathbb{P}_0(\tria;\mathbb{R}^d)$, see \cite[Def.~5.1]{DieningStorn25}, and the error indicator $\eta^2_\tria \coloneqq \sum_{T\in \tria} \eta^2(T)$ reading with $F_\varepsilon^*(Q) \coloneqq \sqrt{(\phi_\varepsilon^*)'(|Q|)/|Q|}\, Q$ for all $Q\in \mathbb{R}^d$
\begin{align}\label{eq:AltEst}
\eta^2(T) \coloneqq \sum_{{\gamma \in \faces(\Omega), \gamma \subset T}} \int_\gamma h_T\, |\llbracket F_\varepsilon^*(\sigma_{h,n})\rrbracket_\gamma|^2 \ds.
\end{align}
We use these error indicators to drive the adaptive scheme in the sense that 
\begin{itemize}
\item we perform another \Kacanov{} iteration if $\eta^2_\textup{iter}$ is the largest error indicator,
\item we increase $\varepsilon_+$ by a factor of 2 if $\eta_+^2$ is the largest error indicator,
\item we decrease $\varepsilon_-$ by a factor of 1/2 if $\eta_-^2$ is the largest error indicator,
\item we perform an adaptive mesh refinement with D\"orfler marking strategy and bulk parameter $\theta = 0.3$ with $\eta_\tria^2$ as refinement indicator if $\eta^2_\tria$ is the largest error indicator.
\end{itemize}
With this strategy we compute numerical approximations of the exact solution. We terminate the computation once  $\dim V_h > 500\,000$.
We then decrease $\varepsilon_-$ by $0.01$ and increase $\varepsilon_+$ by $100$ and perform 50 further \Kacanov{} iterations to obtain an accurate reference solution. This reference solution is used to approximate the error for all discrete approximations with $\dim V_h \leq 50\, 000$. 
\subsection{Experiments}
We solve the $p$-Laplacian with underlying L-shaped domain $\Omega = (-1,1)^2\setminus [0,1)^2$, constant right-hand side $f=2$, and inhomogeneous Dirichlet boundary data $u(x,y) = 1-|y|$ on $\partial \Omega$. We use the adaptive scheme discussed in the previous subsection. Figure~\ref{fig:Exp1} displays the convergence history of the relative error $\lVert F(\nabla u) - F(\nabla u_{h,n})\rVert_{L^2(\Omega)}^2 / \lVert F(\nabla u)\rVert_{L^2(\Omega)}^2$ for the lowest-order Lagrange and Crouzeix--Raviart FEM with exponent $p=1.1$, where $u$ denotes the reference solution computed on a finer grid with enlarged relaxation interval. The rate of convergence seems to be better than $\textup{ndof}^{-1/2}$, but worse than the rate $\textup{ndof}^{-1}$ observed in the linear case $p=2$. This is due to reduced regularity properties, cf.~the discussion in \cite[Sec.~7.2]{BalciDieningStorn23}. The experiment suggests a similar behavior of the Lagrange and Crouzeix--Raviart FEM, with the error of the Lagrange FEM being slightly smaller which is likely due to the fewer degrees of freedom on each mesh. Indeed, the a priori error estimate in Theorem~\ref{thm:AprioriConf} suggests that the additional degrees of freedom of the Crouzeix--Raviart space do not provide more beneficial approximation properties than the Lagrange space. This observation extends to the case $p>2$, illustrated by the convergence history plot in Figure~\ref{fig:Exp2} for $p=10$. More precisely, the Lagrange and Crouzeix--Raviart FEM behave similarly, with the Lagrange FEM leading to slightly better results. Note that the rate of convergence seems to be much better than in Figure~\ref{fig:Exp1}. On the other hand, the error does not decrease monotonically. This is due to the fact that the dual regularized \Kacanov{} scheme used for the computation is solely monotone with respect to the discrete dual energy which is not displayed in the convergence history plot.

\pgfplotstableread{Experiments/AccumulatedDofError_p1.1CR_False.dat}\datatableA
\pgfplotstablecreatecol[
    create col/expr accum={
        \pgfmathaccuma + \thisrow{ndof} 
    }{0} 
]{accumulated_ndof_A}\datatableA
\pgfplotstableread{Experiments/AccumulatedDofError_p1.1CR_True.dat}\datatableB
\pgfplotstablecreatecol[
    create col/expr accum={
        \pgfmathaccuma + \thisrow{ndof}
    }{0}
]{accumulated_ndof_B}\datatableB

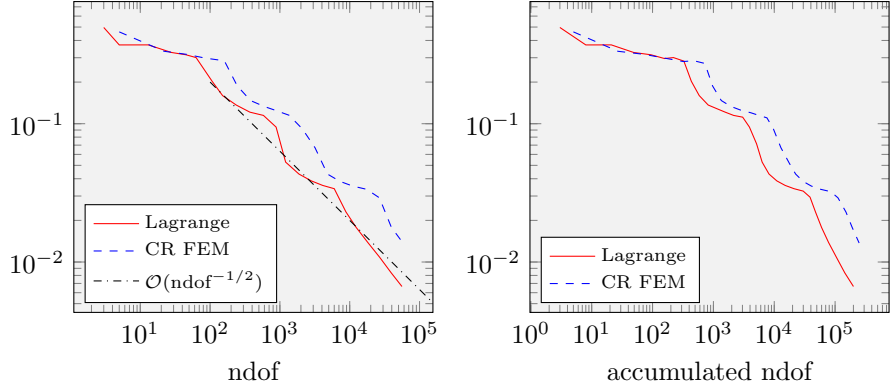
\begin{figure}[ht!]
\begin{tikzpicture}
\begin{axis}[
clip=false,
width=.5\textwidth,
height=.45\textwidth,
ymode = log,
xmode = log,
xlabel = {$\textup{ndof}$},
cycle multi list={\nextlist MyCol1b},
scale = {1},
clip = true,
legend cell align=left,
legend style={legend columns=1,legend pos= south west,font=\fontsize{7}{5}\selectfont}
]
	\addplot table [x=ndof,y=relat_error] {Experiments/Error_p1.1CR_False.dat};
	\addplot table [x=ndof,y=relat_error] {Experiments/Error_p1.1CR_True.dat};
	\legend{{Lagrange},{CR FEM},{$\mathcal{O}(\textup{ndof}^{-1/2})$}};
	\addplot[dash dot,sharp plot,update limits=false] coordinates {(1e2,2e-1) (1e6,2e-3)};
\end{axis}
\end{tikzpicture}
\begin{tikzpicture}
\begin{axis}[
clip=false,
width=.5\textwidth,
height=.45\textwidth,
ymode = log,
xmode = log,
xlabel = {accumulated $\textup{ndof}$},
cycle multi list={\nextlist MyCol1b}, 
scale = {1},
clip = true,
legend cell align=left,
legend style={legend columns=1,legend pos=south west,font=\fontsize{7}{5}\selectfont}
]
	\addplot table [x=accumulated_ndof_A,y=relat_error] {\datatableA};
	\addplot table [x=accumulated_ndof_B,y=relat_error] {\datatableB};
	\legend{{Lagrange},{CR FEM}};
\end{axis}
\end{tikzpicture}
\caption{Convergence history of the relative errors $\lVert F(\nabla u) - F(\nabla u_{h,n})\rVert_{L^2(\Omega)}^2 / \lVert F(\nabla u)\rVert_{L^2(\Omega)}^2$ for the lowest-order Lagrange and Crouzeix--Raviart FEM with $p=1.1$. The left-hand side displays the error of iterates $u_{h,n}$ before a mesh refinement, plotted against the degrees of freedom. The right-hand side displays the error of any iterate $u_{h,n}$ plotted against the accumulated number of degrees of freedom.} \label{fig:Exp1}
\end{figure}%

\pgfplotstableread{Experiments/AccumulatedDofError_p10CR_False.dat}\datatableC
\pgfplotstablecreatecol[
    create col/expr accum={
        \pgfmathaccuma + \thisrow{ndof} 
    }{0} 
]{accumulated_ndof_C}\datatableC

\pgfplotstableread{Experiments/AccumulatedDofError_p10CR_True.dat}\datatableD

\pgfplotstablecreatecol[
    create col/expr accum={
        \pgfmathaccuma + \thisrow{ndof}
    }{0}
]{accumulated_ndof_D}\datatableD

\begin{figure}[ht!]
\begin{tikzpicture}
\begin{axis}[
clip=false,
width=.5\textwidth,
height=.45\textwidth,
ymode = log,
xmode = log,
xlabel = {$\textup{ndof}$},
cycle multi list={\nextlist MyCol1b},
scale = {1},
clip = true,
legend cell align=left,
legend style={legend columns=1,legend pos= south west,font=\fontsize{7}{5}\selectfont}
]
	\addplot table [x=ndof,y=relat_error] {Experiments/Error_p10CR_False.dat};
	\addplot table [x=ndof,y=relat_error] {Experiments/Error_p10CR_True.dat};
		\addplot[dash dot,sharp plot,update limits=false] coordinates {(1e2,1e-1) (1e6,1e-5)};
	\legend{{Lagrange},{CR FEM},{$\mathcal{O}(\textup{ndof}^{-1})$}};
\end{axis}
\end{tikzpicture}
\begin{tikzpicture}
\begin{axis}[
clip=false,
width=.5\textwidth,
height=.45\textwidth,
ymode = log,
xmode = log,
xlabel = {accumulated $\textup{ndof}$},
cycle multi list={\nextlist MyCol1b}, 
scale = {1},
clip = true,
legend cell align=left,
legend style={legend columns=1,legend pos=south west,font=\fontsize{7}{5}\selectfont}
]
	\addplot table [x=accumulated_ndof_C,y=relat_error] {\datatableC};
	\addplot table [x=accumulated_ndof_D,y=relat_error] {\datatableD};
	\legend{{Lagrange},{CR FEM}};
\end{axis}
\end{tikzpicture}
\caption{Convergence history of the relative errors $\lVert F(\nabla u) - F(\nabla u_{h,n})\rVert_{L^2(\Omega)}^2 / \lVert F(\nabla u)\rVert_{L^2(\Omega)}^2$ for the lowest-order Lagrange and Crouzeix--Raviart FEM with $p=10$. The left-hand side displays the error of iterates $u_{h,n}$ before a mesh refinement, plotted against the degrees of freedom. The right-hand side displays the error of any iterate $u_{h,n}$ plotted against the accumulated number of degrees of freedom.}\label{fig:Exp2}
\end{figure}
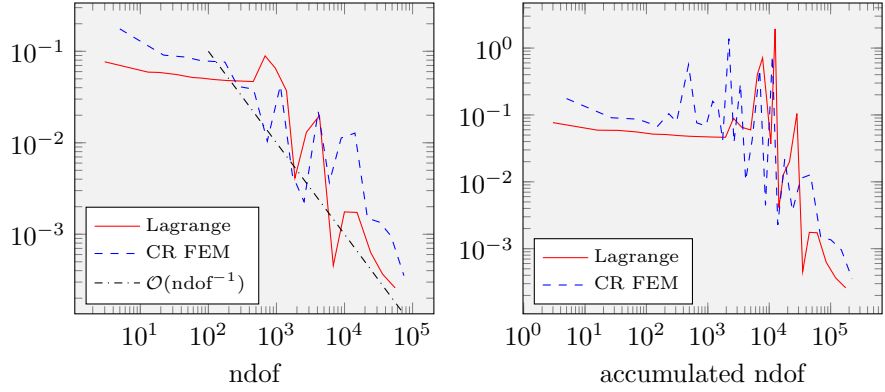%

\section*{Conclusion}
This paper provides the first quasi-optimal approximation result for the Crouzeix--Raviart FEM applied to the $p$-Laplace problem. Our analysis overcomes several challenges, including in particular the treatment of tangential jump contributions. As a byproduct, we prove that the Lagrange FEM enjoys the same localized approximation properties as the Crouzeix--Raviart FEM, a fact further supported by our numerical experiments.

However, in contrast to the adaptive scheme for the conforming case, see for example \cite{DieningFornasierTomasiWank20,BalciDieningStorn23}, the adaptive scheme for the Crouzeix--Raviart FEM currently lacks solid theoretical foundations: a reliable and efficient a posteriori error estimator, as available in the conforming setting \cite{DieningKreuzer08}, is not known. In fact, we observed difficulties for a naive extension of the conforming estimator from \cite{DieningKreuzer08} for large values $p \gg 2$, which necessitated the alternative error indicator in \eqref{eq:AltEst}. This motivates further theoretical investigations aimed at improving the behavior of the adaptive non-conforming approach. Our novel treatment of the tangential jump contributions may constitute an important step towards a rigorous theoretical understanding of the corresponding a posteriori error estimator.

The result is particularly interesting in the regimes $p \ll 2$ and $p \gg 2$, since the equivalence constants in Theorems~\ref{thm:Apriori} and \ref{thm:AprioriConf}, which imply a comparable accuracy of the conforming and non-conforming methods, degenerate as $p$ approaches one or infinity. Hence, even though our experiments indicate a similar behavior of the adaptive conforming and non-conforming schemes, there may still be beneficial properties of the non-conforming method in these regimes that are currently hidden in our analysis due to the degenerating equivalence constants.

A further advantage of the Crouzeix--Raviart FEM is its ability to approximate minimizers in the presence of a Lavrentiev gap \cite{Ortner11,BalciOrtnerStorn22}. A rigorous a priori error analysis for such energies remains an open problem that may benefit from the techniques developed in this paper. Another application of the Crouzeix--Raviart FEM is the computation of lower eigenvalue bounds, as carried out in the linear case in \cite{CarstensenGedicke14}. Extending these results to the $p$-Laplace setting is again an open problem, for which our novel analysis provides promising groundwork.

\subsection*{Acknowledgments.}
The author is grateful to Professor Tabea Tscherpel and Professor Lars Diening for insightful discussions that initiated the research presented here. He also thanks Professor Dietmar Gallistl for carefully reading the manuscript and for identifying an initial inaccuracy in the treatment of the boundary. Last but not least, he thanks the referees, whose constructive suggestions led, among other things, to the inclusion of Remark~\ref{rem:HigherOrder}.
\printbibliography

\end{document}